\documentclass{amsart}

\usepackage{amsfonts}
\usepackage{amsthm}
\usepackage{amstext}
\usepackage{amsmath}
\usepackage{amscd}
\usepackage{amssymb}
\usepackage[mathscr]{eucal}
\usepackage{epsf}

\newtheorem{theorem}{Theorem}
\newtheorem{corollary}[theorem]{Corollary}
\newtheorem{proposition}[theorem]{Proposition}
\newtheorem{lemma}[theorem]{Lemma}

\theoremstyle{definition}
\newtheorem{definition}[theorem]{Definition}
\newtheorem{example}[theorem]{Example}

 \theoremstyle{remark}

\numberwithin{equation}{section}

\newcommand\mL{L\kern-0.08cm\char39}

\newcommand{\Id}{\mathop{\rm Id}}
\newcommand{\Int}{\mathop{\rm Int}}

\newcommand{\card}{\mathop{\rm card}}
\newcommand{\dist}{\mathop{\rm dist}}

\newcommand{\Orb}{\mathop{\rm Orb}}

\newcommand{\SSS}{\mathbb S}
\newcommand{\NNN}{\mathbb N}
\newcommand{\RRR}{\mathbb R}
\newcommand{\DDd}{\mathcal{D}}
\newcommand{\tr}{\mathop{\rm Tr}}
\newcommand{\intr}{\mathop{\rm Intr}}

\newcommand{\abs}[1]{\lvert#1\rvert}

\newcommand{\md}{d\kern-0.035cm\char39\kern-0.03cm}
\newcommand{\mt}{t\kern-0.035cm\char39\kern-0.03cm}
\newcommand{\ml}{l\kern-0.035cm\char39\kern-0.03cm}

\renewcommand{\langle}{[}
\renewcommand{\rangle}{]}

\newcommand{\Fix}{\operatorname{Fix}}
\newcommand{\Per}{\operatorname{Per}}
\newcommand{\Rec}{\operatorname{Rec}}

\begin{document}

\title[Dynamical consequences of a free interval]
{Dynamical consequences of a free interval:\\minimality, transitivity, mixing and topological entropy}

%    Information for the authors
\author[M. Dirb\'ak, \mL. Snoha, V. \v Spitalsk\'y]{Mat\'u\v s Dirb\'ak, \mL{}ubom\'ir Snoha, Vladim\'ir \v Spitalsk\'y}
\address{Department of Mathematics, Faculty of Natural Sciences,
          Matej Bel University, Tajovsk\'eho 40, 974 01 Bansk\'a Bystrica,
          Slovakia}
\email{Matus.Dirbak@umb.sk, Lubomir.Snoha@umb.sk, Vladimir.Spitalsky@umb.sk}
\thanks{The first author was supported by the Slovak Research and Development Agency, grant
APVV LPP-0411-09. The second and the third authors were supported by the the Slovak Research and Development Agency, grant
APVV-0134-10 and by VEGA, grant 1/0978/11.}

%    General info

\dedicatory{Dedicated to our teacher, friend and colleague Alfonz Haviar on the occasion of his retirement.}


\begin{abstract}
We study dynamics of continuous maps on compact metrizable spaces containing a free interval (i.e., an open subset homeomorphic to an open interval). A special attention is paid to
relationships between topological transitivity, weak and strong topological mixing,
dense periodicity and topological entropy as well as to the topological structure
of minimal sets. In particular, a trichotomy for minimal sets and a dichotomy for transitive maps are proved.
\end{abstract}

\maketitle

%%%%%%%%%%%%%%%%%%%%%%%%%%%%%%%%%%%%%%%%%%%%%%%%%%%%%%%%%%%%%%%%%%%%%%%%%%%%%%%
%%%%%%%%%%%%%%%%%%%%%%%%%%%%%%%%%%%%%%%%%%%%%%%%%%%%%%%%%%%%%%%%%%%%%%%%%%%%%%%
%%%%%%%%%%%%%%%%%%%%%%%%%%%%%%%%%%%%%%%%%%%%%%%%%%%%%%%%%%%%%%%%%%%%%%%%%%%%%%%
%%%%%%%%%%%%%%%%%%%%%%%%%%%%%%%%%%%%%%%%%%%%%%%%%%%%%%%%%%%%%%%%%%%%%%%%%%%%%%%
%%%%%%%%%%%%%%%%%%%%%%%%%%%%%%%%%%%%%%%%%%%%%%%%%%%%%%%%%%%%%%%%%%%%%%%%%%%%%%%
\section{Introduction}\label{S:introduction}
One-dimensional dynamics became an object of wide interest in the middle of 1970's,
some 10 years after Sharkovsky's theorem,
when chaotic phenomena were discovered by Li and Yorke in dynamics of interval maps.
As general references one can recommend the
monographs~\cite{CE80}, \cite{BC}, \cite{ALM} and~\cite{dMvS} (several motivations for studying one-dimensional
dynamical systems are discussed in the introduction of ~\cite{dMvS}). The interval and circle dynamics are well understood.
To extend/generalize the results to graph maps is sometimes quite easy, sometimes extremely difficult
(for instance the characterization of the set of periods for graph maps is known only in very special cases).
A good, more than 50 pages long survey of some topics in the dynamics of graph maps can be found in
an appendix in~\cite{ALM} (the second edition). Also the paper~\cite{Blo84} and the paper \cite{Blo}
with its two continuations under the same title are a must for everybody who wishes to study the dynamics of graph maps.

When working in one-dimensional topological dynamics it is natural to try to extend results from the interval to more general spaces.
One usually extends them (or slight modifications of them) first to the circle/trees and then to general graphs (then perhaps also to special kinds
of dendrites or to other one-dimensional continua; in the case of general dendrites often a counterexample can be found).
The present paper suggests that sometimes another approach can be more fruitful --- we show that some important facts from the topological dynamics on the interval/circle work on much more general spaces than graphs, namely on spaces containing an \emph{open} part looking like an interval (we will call it a \emph{free interval}). It seems that the first result indicating that the presence of a free interval might have important dynamical consequences for the whole space was obtained as early as 1988 in~\cite{Kaw}, for two other results see~\cite{AKLS} and \cite{HKO}. However, they were isolated in a sense and apparently have not attracted much attention; a systematic study of the influence of a free interval on dynamical properties of a space has not been done yet. Based on the main results of the present paper, we believe that namely the class of spaces with a free interval is a natural candidate for possible extension of classical results of one-dimensional topological dynamics. Of course, not all of them can be carried over from the interval/circle to spaces with a free interval (and then trees/graphs naturally enter the scene as candidates for possible extension). However, we hope that the elegance of the main results of the present paper (see Theorems~A,~B and~C) justifies our belief.

Recall the terminology which is being used to describe spaces studied in this paper.

\begin{definition}\label{D:free}
Let $X$ be a topological space. We say that
$J$ is a \emph{free interval of $X$} if it is an open subset of $X$ homeomorphic to an open
interval of the real line.
\end{definition}

\begin{definition}\label{D:disc}
Let $X$ be a connected topological space and $J$ be a free interval of $X$. We say that
$J$ is a \emph{disconnecting interval of $X$} if $X\setminus\{x\}$ has exactly two components for every point $x\in J$.
\end{definition}

The definition of a disconnecting interval is taken from \cite{AKLS}.
One can suggest several other ``natural'' definitions of a disconnecting interval. However, we warn the reader that
only some of them are equivalent in the setting of general connected Hausdorff topological spaces. We postpone a discussion
on this to Appendix~2.

Throughout the paper, by an \emph{interval} in $X$ we mean any nonempty subset of $X$ homeomorphic to a (possibly degenerate)
interval in $\RRR$. Of course, all free intervals are intervals.
An open subinterval of a free/disconnecting interval is also free/disconnecting.

An \emph{arc} $A$ is a homeomorphic image
of $[0,1]$. The closure of a free interval of $X$ need not be an arc (say, this closure may look like
the topologist's sine curve). If it \emph{is} an arc, it is called a \emph{free arc} of $X$.
Notice that a space $X$ contains a free arc if and only if it contains a free interval.

Here are all the results on the dynamics of continuous maps on general spaces with a free or a disconnecting interval, which we have found in the literature.

\begin{itemize}
\item No Peano continuum with a free arc admits an expansive homeomorphism (\cite{Kaw}). For generalizations see~\cite{MShi}, \cite{SW}.
\item If $X$ is a connected space with a disconnecting interval and a continuous map $f:X\to X$
      is transitive then the set of all periodic points of $f$ is dense in $X$ (\cite{AKLS}).
\item Let $X$ be a compact metrizable space with a free interval. Then every totally transitive continuous map $f: X\to X$ with dense periodic points is strongly mixing (\cite{HKO}).
\end{itemize}

We suggest to study the dynamics of continuous maps on spaces with a
free or a disconnecting interval systematically and compare the
results with those working on graphs or trees. For the present paper
we have chosen those problems which seemed to us most
important/interesting. The obtained results are potentially
applicable in the study of dynamics on some classes of
one-dimensional continua and some of them are even surprisingly
nice.

Let us now summarize the main results of this paper. The first two
theorems deal with minimal systems and minimal sets, the third one
is a dichotomy for transitive maps. Note that all the maps in the
paper are assumed to be continuous. For definitions see
Sections~\ref{S:preliminaries} and \ref{S:rpd} (here we recall only
those which are not widely known).

%%%%%%%%%%%%%%%%%%%%%%%%%%% MAIN THEOREMS %%%%%%%%%%%%%%%%%%%%%%%%%%%
\newcommand{\rotationX}{$X$ is a disjoint union of finitely many circles,
$X=\bigoplus_{i=0}^{n-1}\SSS^1_i$, which are cyclically permuted by $f$
and, on each of them, $f^n$ is topologically conjugate to the same irrational rotation}

\newcommand{\rotationM}{$M$ is a disjoint union of finitely many circles,
$M=\bigoplus_{i=0}^{n-1}\SSS^1_i$, which are cyclically permuted by $f$
and, on each of them, $f^n$ is topologically conjugate to the same irrational rotation}

\newcommand{\theoremA}{
\renewcommand{\thetheorem}{A}
\begin{theorem}[\bf{Minimal spaces with a free interval}]
Let $X$ be a compact metrizable space with a free interval $J$ and let $f:X\to X$
be a minimal map. Then \rotationX{}.
\end{theorem}
\addtocounter{theorem}{-1}
\renewcommand{\thetheorem}{\arabic{theorem}}
}

\newcommand{\theoremB}{
\renewcommand{\thetheorem}{B}
\begin{theorem}[\bf{Trichotomy for minimal sets}]
Let $X$ be a compact metrizable space with a free interval $J$ and let $f:X\to X$
be a continuous map. Assume that $M$ is a minimal set for $f$ which intersects
$J$. Then exactly one of the following three statements holds.
\begin{enumerate}
\item[(1)] $M$ is finite.
\item[(2)] $M$ is a nowhere dense cantoroid.
\item[(3)] $M$ is a disjoint union of finitely many circles.
\end{enumerate}
\end{theorem}
\addtocounter{theorem}{-1}
\renewcommand{\thetheorem}{\arabic{theorem}}
}

\newcommand{\theoremC}{
\renewcommand{\thetheorem}{C}
\begin{theorem}[\bf{Dichotomy for transitive maps}]%\label{dichotomy}
Let $X$ be a compact metrizable space with a free interval and let
$f:X\to X$ be a transitive map. Then exactly one of the following
two statements holds.
\begin{enumerate}
\item[(1)] The map $f$ is relatively strongly mixing, non-invertible,
 has positive topological entropy and dense periodic points.
\item[(2)] The space \rotationX{}.
\end{enumerate}
\end{theorem}
\addtocounter{theorem}{-1}
\renewcommand{\thetheorem}{\arabic{theorem}}
}

\newcommand{\theoremAref}{Theorem~A}
\newcommand{\theoremBref}{Theorem~B}
\newcommand{\theoremCref}{Theorem~C}
\newcommand{\theoremABrefs}{Theorems~A and B}
\newcommand{\theoremABCrefs}{Theorems~A--C}

\medskip

It is well known that a graph admits a minimal map if and only if
it is a finite union of disjoint circles (see \cite{Blo84}, cf. \cite{BHS}). Our first result generalizes
this fact.

\theoremA

Also characterization of minimal sets on graphs is well known: these are finite sets, Cantor sets and finite disjoint unions
of circles, see \cite{BHS}. In \cite{BDHSS} this was generalized to local dendrites. In that connection the notion of a cantoroid
was introduced.
According to \cite{BDHSS}, a \emph{cantoroid} is a compact metrizable space without isolated points in which
degenerate (connected) components are dense.

The following theorem describes the minimal sets which intersect a free interval of a space.

\theoremB

Note that none of these three conditions can be removed. Already on the circle each of the three cases can occur; here in the case (2) $M$ is a Cantor set.
Moreover, if $X$ is a (local) dendrite containing a free interval as well as a non-degenerate nowhere dense subcontinuum
then there is a cantoroid different from a Cantor set which intersects that free interval and contains that subcontinuum
and so by \cite{BDHSS} it is a minimal set for some continuous selfmap of $X$. Hence, in the case (2) we cannot replace ``cantoroid'' by
``Cantor set''.

Obviously, in the case (2), $M\cap J$ is a union of countably many Cantor sets (sometimes such a set is called a \emph{Mycielski set}) and if $M\subseteq J$
then $M$ is a Cantor set.
In the case (3), $M$ contains the whole free interval $J$ and by \theoremAref{}, applied to $f|_M$, we get that
$M$ is a disjoint union of finitely many circles which are cyclically permuted by $f$.
Then, on each of these circles, the corresponding iteration of $f$ is topologically conjugate to the same irrational rotation.

\medskip

To state our next theorem we recall the following notion.
If $\DDd=(D_0,\dots,D_{n-1})$ is a regular periodic decomposition for $f$ (see Section~\ref{S:rpd})
we say, according to \cite{B}, that $f$ is \emph{strongly mixing relative to $\DDd$} if
$f^n$ is strongly mixing on each of the sets $D_i$. Also, we say that $f$ is \emph{relatively strongly mixing}
if it is strongly mixing relative to some of its regular periodic decompositions.

\theoremC

Three obvious remarks to \theoremCref{} can be made. First, if $X$ is a continuum with a \emph{disconnecting} interval
then necessarily the case (1) holds. Second, if $f$ is \emph{totally} transitive then
$X$ is a continuum (otherwise $X$ has $k\ge 2$ components permuted by $f$, a contradiction), in the case (1) the map $f$ is strongly mixing
(if not then the RPD from the definition of relative strong mixing has $m\ge 2$ elements and $f^m$ is not transitive)
and in the case (2) the space $X$ is a circle. Third, none of the two conditions in Theorem~C is superfluous. To see it, just consider the tent map on the interval and
an irrational rotation of the circle.

Kwietniak in \cite{Kwi11} independently obtained the following result
which follows from our \theoremCref{}:
Any weakly mixing map on a compact metrizable space  with  a free interval
is strongly mixing, has dense periodic points and positive entropy.

By \cite{Blo84}, on \emph{graphs} the following dichotomy,
stronger than that in our Theorem~C,
is true: A transitive system on a (not necessarily connected) graph either has the relative specification property
(hence also all the properties from the case (1))
or the case (2) from Theorem~C holds.
It is a challenge to solve the problem
whether this stronger dichotomy is true for every compact metrizable space
with a free interval.

\medskip

For some other results which are worth of mentioning and are not covered by
\theoremABCrefs{}, see Theorems~\ref{per.rec.pts}, \ref{per.rec.pts2}
and Corollary~\ref{no.ToP.sys}.

\medskip

Let us state some applications of our main results.

The Warsaw circle is one of the simplest continua which are not locally connected.
The dynamics on this particular space has been studied since 1996, see \cite{Warsaw96}.
The main result of the recent paper \cite{Warsaw08} says that
every transitive map on the Warsaw circle $W$ has a horseshoe (hence positive topological entropy) and dense periodic points
and is strongly mixing. Our \theoremCref{} gives, for granted, a result which is on one hand only slightly weaker and, on the other hand,
works on a whole class of spaces including the Warsaw circle. Namely if $X\ne\SSS^1$ is a continuum with a free interval
then every transitive map on $X$ has positive topological entropy and dense periodic points
and every totally transitive map on $X$ is strongly mixing.

Baldwin in \cite{Baldwin01} asked whether every transitive map on a dendrite has positive topological entropy.
The problem is still open but notice that our \theoremCref{} implies that
the answer is positive for dendrites whose branch points are not dense.

\medskip

The paper is organized as follows. In Section~\ref{S:preliminaries} we recall
some definitions and known facts from topological dynamics.
In Section~\ref{S:rpd} we apply the theory of regular periodic decompositions for transitive maps,
developed by Banks in \cite{B} in the setting of general topological spaces, to spaces with free intervals.
In Section~\ref{S:periodic-recurrent}
we study connections between recurrent and periodic points on spaces
with free intervals.
In Section~\ref{S:periodic-density}
we prove some conditions sufficient for the density
of (eventually) periodic points.
Section~\ref{S:minimality} deals with minimal systems and
minimal sets on spaces with free intervals; it contains proofs of \theoremABrefs{}.
Then in Sections~\ref{S:dense-periodicity}, \ref{S:entropy} and \ref{S:mixing}
we study, respectively, dense periodicity, topological entropy and strong mixing for transitive maps.
The proof of \theoremCref{} is contained in Section~\ref{S:mixing}.
A proof of Theorem~\ref{per.rec.pts2}, which is a generalization of
a theorem from~\cite{MS}, is given in Appendix~1.
Finally, in Appendix~2 we discuss relations between several ``natural'' definitions of a disconnecting interval,
see Proposition~\ref{disc.intrs}.
Moreover, the main results of both appendices are used in the proof of Lemma~\ref{perpt.freeint}.

%%%%%%%%%%%%%%%%%%%%%%%%%%%%%%%%%%%%%%%%%%%%%%%%%%%%%%%%%%%%%%%%%%%%%%%%%%%%%%%
%%%%%%%%%%%%%%%%%%%%%%%%%%%%%%%%%%%%%%%%%%%%%%%%%%%%%%%%%%%%%%%%%%%%%%%%%%%%%%%
%%%%%%%%%%%%%%%%%%%%%%%%%%%%%%%%%%%%%%%%%%%%%%%%%%%%%%%%%%%%%%%%%%%%%%%%%%%%%%%
%%%%%%%%%%%%%%%%%%%%%%%%%%%%%%%%%%%%%%%%%%%%%%%%%%%%%%%%%%%%%%%%%%%%%%%%%%%%%%%
%%%%%%%%%%%%%%%%%%%%%%%%%%%%%%%%%%%%%%%%%%%%%%%%%%%%%%%%%%%%%%%%%%%%%%%%%%%%%%%
\section{Preliminaries}\label{S:preliminaries}

Here we briefly recall all the notions and results which will be needed
in the rest of the paper.

We write $\NNN$ for the set of positive integers $\{1,2,3,\dots\}$
and $I$ for the unit interval $[0,1]$.
A \emph{space} means a topological space.
If $X$ is a connected space and $x\in X$ is a cut point of $X$ (i.e. $X\setminus\{x\}$ is not connected)
we also say that $x$ \emph{cuts} $X$
(by saying that \emph{$x$ cuts $X$ into two components}
we mean that $X\setminus\{x\}$ has exactly two components).
A \emph{continuum} is a connected compact metrizable space.
By $\Int A$, $\overline{A}$, $\partial{A}$ and $\card A$ we denote the interior, closure, boundary and
cardinality of $A$.
For definitions of a cantoroid, a free interval/arc and a disconnecting interval see Section~\ref{S:introduction}.
If $J$ is a free interval/arc then we always assume that one of two natural orderings
(induced by usual orderings on a real interval) is chosen and denoted by $\prec$.
We will use the usual notations for subintervals of $J$, say we write
$[a,b) = \{x\in J:\ a\preceq x \prec b\}$ for $a\prec b$ in $J$.
Throughout the paper no distinction is made between a point $x$ and the singleton $\{x\}$.

A \emph{(discrete) dynamical system} is a pair $(X,f)$ where $X$ is a topological
space and $f:X\to X$ is a (possibly non-invertible)
continuous map. The iterates of $f$ are defined by $f^0=\Id_X$ (the identity map on $X$) and
$f^n=f^{n-1}\circ f$ for $n\ge 1$.
The \emph{orbit} of $x$ is the set $\Orb_f(x) = \{f^n(x):\ n\ge 0\}$.
A point $x\in X$ is a \emph{periodic point} of $f$ if $f^n(x)=x$ for some $n\in\NNN$.
The smallest such $n$ is called the \emph{period} of $x$.
If $f^m(x)$ is periodic for some $m\in\NNN$ we say that $x$ is \emph{eventually periodic}.
A point $x$ is \emph{recurrent} if for every neighborhood $U$ of $x$ there are arbitrarily large $n$
with $f^n(x)\in U$.
By $\Per(f)$ or $\Rec(f)$ we denote the set of all periodic or recurrent points of $f$, respectively.

A system $(X,f)$ is \emph{minimal} if every orbit is dense. A set $A\subseteq X$ is \emph{minimal}
for $f$ if it is nonempty, closed, \emph{$f$-invariant} (i.e. $f(A)\subseteq A$) and
$(A,f|_A)$ is a minimal system. A system $(X,f)$ is called \emph{totally minimal} if
$(X,f^n)$ is minimal for every $n=1,2,\dots$.

A dynamical system $(X,f)$ is \emph{(topologically) transitive} if
for every non-empty open sets $U,V\subseteq X$ there is $n\in\NNN$ such that
$f^n(U)\cap V\ne\emptyset$. A point whose orbit is dense is called a \emph{transitive point};
points which are not transitive are called \emph{intransitive}.
The set of transitive or intransitive points of $f$ is denoted by $\tr(f)$ or $\intr(f)$, respectively.
If $(X,f^n)$ is transitive for all $n\in\NNN$ we say that $(X,f)$ is called \emph{totally transitive}.
If $(X\times X, f\times f)$ is transitive then $(X,f)$
is called \emph{(topologically) weakly mixing}.
It is well known that if $X$ is a compact metrizable space and $(X,f)$ is weakly mixing then for every $n\ge 1$ the system
$(X\times \dots \times X, f\times \dots \times f)$ ($n$-times) is topologically transitive.
A system $(X,f)$ is
\emph{(topologically) strongly mixing}
provided for every non-empty open sets $U,V\subseteq X$ there is $n_0\in\NNN$ such that
$f^n(U)\cap V\ne\emptyset$ for every $n\ge n_0$.
Instead of saying that a system $(X,f)$ has some of the defined properties (minimality, transitivity, \dots)
we also say that the map $f$ itself has this property.

Finally, the \emph{topological entropy} of $(X,f)$ will be denoted by $h(f)$. We assume that the reader is familiar with
Bowen's definition of topological entropy which uses the notion of $(n,\varepsilon)$-separated sets;
see for instance \cite{ALM}.

\begin{lemma}[\cite{B}]\label{TT+DPimplWM}
Let $X$ be a topological space and $f:X\to X$ a continuous map.
If $f$ is totally transitive and has dense set of periodic points,
then $f$ is weakly mixing.
\end{lemma}

\begin{lemma}[\cite{AKLS}]\label{exist.per.pt}
Let $X$ be a connected topological space with a disconnecting interval $J$ and
let $f:X\to X$ be a continuous map. Assume that there exist
$x,y\in J$ and $n,m\geq 1$ such that $f^n(x), f^m(y)\in J$,
$f^n(x)\prec x$ and $f^m(y)\succ y$. Then $f$ has a periodic
point in the convex hull of $\{ x,y,f^n(x),f^m(y)\}$.
\end{lemma}

\begin{lemma}[\cite{AKLS}]\label{denseperiod}
Let $X$ be a connected topological space with a disconnecting interval and let
$f:X\to X$ be a transitive map. Then the set of all periodic points of
$f$ is dense in $X$.
\end{lemma}

Our \theoremCref{} is stronger than the following result from \cite{HKO};
however, we will use it in the course of proving \theoremCref{}.

\begin{lemma}[\cite{HKO}]\label{TT+DPimplSM}
Let $X$ be a compact metrizable space with a free interval.
Then every totally transitive map $f:X\to X$ with dense periodic points
is strongly mixing.
\end{lemma}

We will use several topological properties of minimal systems
which we summarize in the following lemmas.

A set $G\subseteq X$ is said to be a \emph{redundant open set for a map $f: X \to
X$} if $G$ is nonempty, open and $f(G)\subseteq f(X\setminus G)$
%$f(X)= f(X\setminus G)$
(i.e., its removal from the domain of $f$ does not change the image
of $f$).

\begin{lemma}[\cite{KST}]\label{red.open.set}
Let $X$ be a compact Hausdorff space and $f: X \to X$ continuous.
Suppose that there is a redundant open set for $f$. Then the system $(X,f)$ is not minimal.
\end{lemma}

\begin{lemma}[\cite{KST}]\label{feebly.open}
Let $X$ be a compact Hausdorff space and let $f:X\to X$ be
a minimal map. Then $f$ is feebly open, i.e. $f$
sends nonempty open sets to sets with nonempty interior.
\end{lemma}

\begin{lemma}[\cite{KST}]\label{unique.preimage}
Let $X$ be a compact metrizable space and let $f:X\to X$ be
a minimal map. Then the set $\{x\in X:\ \card f^{-1}(x)=1\}$
is a $G_\delta$-dense set in $X$.
\end{lemma}

Notice that a minimal map $f$ on a compact Hausdorff space $X$ is necessarily
surjective and so $f^{-1}(x)$ is nonempty for every $x\in X$.

Minimal maps preserve several important topological properties of
sets both forward and backward. We will explicitly use the following
result.

\begin{lemma}[\cite{KST}]\label{image.residual}
Let $X$ be a compact Hausdorff space and let $f:X\to X$ be a minimal
map. If $R$ is a residual subset of $X$ then so is $f(R)$.
\end{lemma}

If $(X,f)$ is a dynamical system and $x\in X$ then any subset
$\{x_n\,:\,n\geq0\}$ of $X$ satisfying $x_0=x$ and $f(x_{n+1})=x_n$
($n\geq0$) is called a \emph{backward orbit} of $x$ (under $f$).
We will use the following fact (see the proof of \cite[Theorem 2.8]{KST}
or \cite{Mal11};
in fact also the converse is true -- for continuous selfmaps of compact metrizable spaces
the density of all backward orbits implies
minimality).

\begin{lemma}[\cite{KST}]\label{dense.backward}
Let $X$ be a compact metrizable space and let $f:X\to X$ be a continuous
map. If $f$ is minimal then all backward orbits are dense.
\end{lemma}

The following result easily follows from \cite[Theorem~3.1]{Ye}.

\begin{lemma}[\cite{Ye}]\label{tot.minimal}
Let $X$ be a compact Hausdorff space and let $f:X\to X$ be a minimal map.
If $X$ is connected then $f$ is totally minimal.
\end{lemma}

The following classical result is an immediate consequence
of \cite{Kin} and \cite{Silv} (proved also in \cite{KS}).

\begin{lemma}[\cite{Kin}, \cite{Silv}]\label{intr.pts}
Let $X$ be a compact metrizable space without isolated points and let $f:X\to X$ be a continuous
map.
Then one of the following holds:
\begin{enumerate}
  \item[(1)] $\tr(f)=\emptyset$ and $\intr(f)=X$;
  \item[(2)] $\tr(f)$ is $G_\delta$-dense and $\intr(f)$ is either empty (i.e. the system is minimal) or dense
  (then the system is transitive non-minimal).
\end{enumerate}
\end{lemma}

\begin{lemma}[\cite{D32}]\label{irr.rot}
If $f$ is a transitive homeomorphism of a circle then it is conjugate to an
irrational rotation.
\end{lemma}

%%%%%%%%%%%%%%%%%%%%%%%%%%%%%%%%%%%%%%%%%%%%%%%%%%%%%%%%%%%%%%%%%%%%%%%%%%%%%%%
%%%%%%%%%%%%%%%%%%%%%%%%%%%%%%%%%%%%%%%%%%%%%%%%%%%%%%%%%%%%%%%%%%%%%%%%%%%%%%%
%%%%%%%%%%%%%%%%%%%%%%%%%%%%%%%%%%%%%%%%%%%%%%%%%%%%%%%%%%%%%%%%%%%%%%%%%%%%%%%
%%%%%%%%%%%%%%%%%%%%%%%%%%%%%%%%%%%%%%%%%%%%%%%%%%%%%%%%%%%%%%%%%%%%%%%%%%%%%%%
%%%%%%%%%%%%%%%%%%%%%%%%%%%%%%%%%%%%%%%%%%%%%%%%%%%%%%%%%%%%%%%%%%%%%%%%%%%%%%%
\section{Regular periodic decompositions for transitive maps}\label{S:rpd}
In this section we study regular periodic decompositions
for transitive maps on spaces with a free interval.
We begin by reviewing some results from \cite{B}.

Let $X$ be a topological space. Recall that a set $D\subseteq
X$ is \emph{regular closed} it it is the closure of its
interior or, equivalently, if it equals the closure of an
open set.

Now let $f:X\to X$ be a continuous map. A \emph{regular
periodic decomposition} (briefly RPD) for $f$ is a finite sequence
$\mathcal D=(D_0,\dots ,D_{m-1})$ of regular closed subsets
of $X$ covering $X$ such that $f(D_i)\subseteq D_{i+1 (\text{mod}\,m)}$ for
$0\leq i\leq m-1$ and $D_i\cap D_j$ is nowhere dense in $X$
for $i\ne j$. The integer
$m$ is called the \emph{length} of $\mathcal D$.
From the latter condition in the definition and from
the fact that the boundary of a regular closed set is nowhere dense,
we get, respectively,
\begin{enumerate}
  \item[(RPD1)] $\Int (D_i)\cap D_j=\emptyset$ for $i\ne j$;
  \item[(RPD2)] the boundary of each $D_i$ is nowhere dense.
\end{enumerate}

Now let
$\mathcal D=(D_0,\dots ,D_{m-1})$ be an RPD for a
\emph{transitive} map $f$. Then, by \cite[Lemma~2.1 and Theorem~2.1]{B},
\begin{enumerate}
  \item[(RPD3)] $\overline{f^l(D_i)}=D_{i+l (\text{mod}\,m)}$
   for $0\leq i\leq m-1$ and $l\ge 0$;
  \item[(RPD4)] $f^{-l}(\Int(D_i))\subseteq \Int\left(D_{i-l (\text{mod}\,m)}\right)$
   for $0\leq i\leq m-1$ and $l\ge 0$;
  \item[(RPD5)] $D=\bigcup_{i\ne j}D_i\cap D_j$
is closed, $f$-invariant and nowhere dense;
  \item[(RPD6)] $f^m$ is transitive on each $D_i$.
\end{enumerate}

If all the sets of $\mathcal{D}$ are connected we say that $\mathcal{D}$ is \emph{connected}.
By (RPD3) and the fact that the continuous image of a connected set is connected we get that
\begin{enumerate}
  \item[(RPD7)] if one of the sets $D_i$ is connected then $\mathcal{D}$ is connected.
\end{enumerate}

\begin{lemma}[\cite{B}, Corollary~2.1]\label{exist.RPD}
Let $X$ be a topological space and $f$ be
a transitive map on $X$ with $f^n$ non-transitive for some
$n\geq 2$. Then $f$ has a regular periodic decomposition
of length dividing $n$.
\end{lemma}

Assume now that $\mathcal D=(D_0,\dots ,D_{m-1})$ and
$\mathcal C=(C_0,\dots ,C_{n-1})$ are RPD's for $f$.
We say that $\mathcal C$ \emph{refines}  (or is a \emph{refinement} of)
$\mathcal D$ if
every $C_i$ is contained in some $D_j$.
Then each element of $\DDd$ contains the same number of elements of $\mathcal C$,
so $n$ is a multiple of $m$.

The following is a slight generalization of \cite[Theorem~6.1]{B}; it will be used in the
proof of Lemma~\ref{finite.DI}.

\begin{lemma}\label{con.refin}
Let $X$ be a topological space containing a
nonempty open locally connected subset.
Let $f$ be a transitive map on $X$. Then every regular
periodic decomposition for $f$ has a connected refinement.
\end{lemma}

\begin{proof}
Let $J\ne\emptyset$ be an open locally connected subset of $X$ and
let $\mathcal{D}=\{D_0,\dots,D_{n-1}\}$ be a regular periodic decomposition
for $f$. By (RPD2)
the union of interiors of $D_i$ is dense, so at least one of them
--- say $\Int{D_0}$ --- intersects $J$.
Let $C_0$ be a
connected component of $D_0$ the interior of which intersects $J$
(use that $J$ is locally connected).
By (RPD6), $f^n|_{D_0}: D_0\to D_0$ is transitive. It follows, since $C_0$ is a component
of $D_0$ with nonempty interior, that the set $D_0$ has finitely many components
$C_0,C_1,\dots,C_{m-1}$ ($m\ge 1$) which are permuted by $f^n$ and
$\mathcal{C}=\{C_0,C_1,\dots,C_{m-1}\}$ is a regular periodic decomposition for
$f^n|_{D_0}$.
Let $G =\Int(C_0)$ (interior in $X$, not in $D_0$).
By \cite[Lemma~3.3]{B} we obtain that
$\mathcal{E}=\{\overline{G}, \overline{f^{-(mn-1)}(G)}, \dots, \overline{f^{-(mn-2)}(G)}, \overline{f^{-1}(G)}\}$
is an RPD for $f$. By (RPD4) it is a refinement of $\mathcal D$. Since $\overline{G}=C_0$ is connected,  by (RPD7) we get that
$\mathcal{E}$ is connected.
\end{proof}

Given a transitive map $f$, by \cite{B} the set of all $m\in\NNN$
such that $f$ admits an RPD of length $m$ is called
the \emph{decomposition ideal} of $f$
(the use of the term ``ideal'' is justified by the fact that this set is an ideal
in the lattice of positive integers ordered by divisibility).
If the decomposition ideal of $f$ is finite then there is
an RPD of $f$ of maximal length.
Such an RPD (which
is by \cite[Theorem~2.2]{B} unique up to cyclic permutations of its elements)
is called a \emph{terminal decomposition} of $f$.

\begin{lemma}[\cite{B}, Theorem~3.1]\label{termin.RPD}
Let $X$ be a topological space and $f$ be
a transitive map on $X$. Assume that $\mathcal D=(D_0,
\dots ,D_{m-1})$ is a regular periodic decomposition
for $f$. Then $\mathcal D$ is terminal if and only
if $f^m|_{D_i}$ is totally transitive for $0\leq i
\leq m-1$.
\end{lemma}

Now we study regular periodic decompositions for
transitive maps on spaces with a free interval.

\begin{lemma}\label{finite.DI}
Let $X$ be a topological space with a free interval
$J$ and $f$ be a transitive map on $X$. Assume that
$f$ has a periodic point $x$ in $J$ with period $p$.
Then every regular periodic decomposition $\mathcal D$ for $f$
has length at most $2p$. In particular, $f$ has a terminal RPD.
\end{lemma}
\begin{proof}
Fix an RPD $\mathcal D=(D_0,\dots ,D_{m-1})$ for $f$.
By Lemma~\ref{con.refin} we may assume that $\mathcal D$ is
connected. Two cases are possible:
\begin{enumerate}
 \item[(1)] $x\in \Int(D_i)$ for some $i$;
 \item[(2)] $x\in D_i\cap D_j$ for some $i\ne j$.
\end{enumerate}
In the case (1), the periodicity of $x$ and (RPD4) give
$$
x\in f^{-p}(\Int(D_i))\subseteq
\Int(D_{i-p\,(\text{mod}\,m)}).
$$
Thus $\Int(D_i)\cap \Int(D_{i-p\,(\text{mod}\,m)})\neq \emptyset$
and so $m|p$ by (RPD1). Consequently $m\leq p\leq 2p$.

In the case (2), since $x\in J$ and $D_i$, $D_j$ are connected,
we may choose the ordering in $J$ in such a way that there are
$a\prec x\prec b$ in $J$ with $(a,x]\subseteq D_i$ and
$[x,b)\subseteq D_j$.
Put $L=(a,x)$ and $R=(x,b)$.
Since $f$ is transitive neither $f^p(L)$ nor $f^p(R)$ is a singleton.
However, $f^p(x)=x$ and so
$f^p(z)\in (a,b)$ for all $z$ sufficiently
close to $x$. Hence
\begin{equation}\label{EQ:LR}
 f^p(L)\cap (L\cup R)\neq
 \emptyset
 \qquad\text{and}\qquad
 f^p(R)\cap (L\cup R)\neq \emptyset.
\end{equation}

If $f^p(L)\cap L\neq \emptyset$ then we have
$$
\emptyset \neq \Int(D_i)\cap f^p(D_i)\subseteq \Int(D_i)\cap
D_{i+p\,(\text{mod}\,m)},
$$
whence $m|p$ by (RPD1) and so $m\leq p\leq 2p$. The same inequality $m\le p$ is obtained if $f^p(R)\cap R \ne \emptyset$.
It remains to consider the situation when $f^p(L)\cap L=\emptyset=f^p(R)\cap R$,
 $f^p(L)\cap R \ne\emptyset$ and $f^p(R)\cap L \ne\emptyset$.
 Hence, since $f^p(x)=x$, there are $a\prec a'\prec x\prec b'\prec b$ such that
 for the non-degenerate sets $f^p(L)$ and $f^p(R)$ we have
 $f^p(L)\supseteq R':=(x,b')$
 and $f^p(R)\supseteq L':=(a',x)$.
 Since (\ref{EQ:LR}) obviously holds for $L',R'$ instead of $L,R$ we get that
 $f^{2p}(L)\cap L\ne\emptyset$. Analogously as we obtained above $m|p$ when $f^p(L)\cap L\ne\emptyset$
 now we get $m|2p$
and so $m\leq 2p$.
\end{proof}

%%%%%%%%%%%%%%%%%%%%%%%%%%%%%%%%%%%%%%%%%%%%%%%%%%%%%%%%%%%%%%%%%%%%%%%%%%%%%%%
%%%%%%%%%%%%%%%%%%%%%%%%%%%%%%%%%%%%%%%%%%%%%%%%%%%%%%%%%%%%%%%%%%%%%%%%%%%%%%%
%%%%%%%%%%%%%%%%%%%%%%%%%%%%%%%%%%%%%%%%%%%%%%%%%%%%%%%%%%%%%%%%%%%%%%%%%%%%%%%
%%%%%%%%%%%%%%%%%%%%%%%%%%%%%%%%%%%%%%%%%%%%%%%%%%%%%%%%%%%%%%%%%%%%%%%%%%%%%%%
%%%%%%%%%%%%%%%%%%%%%%%%%%%%%%%%%%%%%%%%%%%%%%%%%%%%%%%%%%%%%%%%%%%%%%%%%%%%%%%
\section{Periodic-recurrent property}\label{S:periodic-recurrent}

In a system $(X,f)$ always $\Per(f)\subseteq \Rec(f)$.
The sets $\Per(f)$ and $\Rec(f)$ need not be closed.
When $\overline{\Per(f)} = \overline{\Rec(f)}$
for every continuous map $f$ on $X$, we speak on the \emph{periodic-recurrent property}
of the space $X$. Some one-dimensional spaces do have this property.
In \cite{I} dendrites with periodic-recurrent property have been characterized.
For the history of the investigation of this property see \cite{I} and \cite{MS}.
A space with a disconnecting interval $J$ need not be one-dimensional
but the periodic-recurrent property, relatively in $J$, still holds.

\begin{theorem}\label{per.rec.pts}
Let $X$ be a connected topological space with a disconnecting interval $J$
and let $f:X\to X$ be a continuous map. Then
$$
\overline{\Rec(f)}\cap J = \overline{\Per(f)}\cap J.
$$
\end{theorem}
\begin{proof}
Only one inclusion needs a proof.
Moreover, it is sufficient to show that $\Rec(f)\cap J \subseteq \overline{\Per(f)}\cap J$
(it is elementary to check that then also $\overline{\Rec(f)}\cap J \subseteq \overline{\Per(f)}\cap J$).
So fix a recurrent point $r\in J$ and consider any open interval $J'$ with
$r\in J'\subseteq J$. We show that $\Per(f)\cap J'\neq \emptyset$. If $r$ itself
is periodic then there is nothing to prove. So assume that $r\notin \Per(f)$.
There are positive integers $n,m$ such that $f^m(r), f^{m+n}(r)\in J'$
and either $r\prec f^{m+n}(r)\prec f^m(r)$ or $f^m(r)\prec f^{m+n}(r)\prec r$.
Without loss of generality we may assume that the first possibility holds.
We thus have $f^m(r)\succ r$ and $f^n(f^m(r))\prec f^m(r)$. By
Lemma~\ref{exist.per.pt}, $f$ has a periodic point in $J'$.
\end{proof}

If $U$ is an open set in a topological space $X$ and $A$ is any set in $X$
then $\overline{\overline{A}\cap U}=\overline{{A}\cap U}$.
Therefore it follows from our theorem that
$$
\overline{\Rec(f)\cap J} = \overline{\Per(f)\cap J}
$$
which is the equality of two sets which are not necessarily subsets of $J$.
One can see that both equalities are equivalent.

In general,
Theorem~\ref{per.rec.pts} is no longer valid if $J$ is assumed to be a free interval
rather than a disconnecting one (consider an irrational rotation of the circle).
One can deduce from \cite{Blo} that for graph maps
a weaker form of the periodic-recurrent property holds, namely
$\overline{\Rec(f)}=\Rec(f) \cup \overline{\Per(f)}$.
Recently Mai and Shao  gave in \cite{MS} a
different proof of this fact. Their idea
can be used to show that such an equality holds
(relatively) in every free interval, see Theorem~\ref{per.rec.pts2}. Though the proof
is pretty similar to that from \cite{MS},
we include it into the appendix because
the result is crucial for Lemma~\ref{perpt.freeint}.

\begin{theorem}\label{per.rec.pts2}
Let $X$ be a topological space with a free interval $J$ and let $f:X\to X$ be a continuous map. Then
$$
\overline{\Rec(f)}\cap J = \left[\Rec(f) ~ \cup ~ \overline{\Per(f)}\right]\cap J.
$$
\end{theorem}

%%%%%%%%%%%%%%%%%%%%%%%%%%%%%%%%%%%%%%%%%%%%%%%%%%%%%%%%%%%%%%%%%%%%%%%%%%%%%%%
%%%%%%%%%%%%%%%%%%%%%%%%%%%%%%%%%%%%%%%%%%%%%%%%%%%%%%%%%%%%%%%%%%%%%%%%%%%%%%%
%%%%%%%%%%%%%%%%%%%%%%%%%%%%%%%%%%%%%%%%%%%%%%%%%%%%%%%%%%%%%%%%%%%%%%%%%%%%%%%
%%%%%%%%%%%%%%%%%%%%%%%%%%%%%%%%%%%%%%%%%%%%%%%%%%%%%%%%%%%%%%%%%%%%%%%%%%%%%%%
%%%%%%%%%%%%%%%%%%%%%%%%%%%%%%%%%%%%%%%%%%%%%%%%%%%%%%%%%%%%%%%%%%%%%%%%%%%%%%%
\section{Density of (eventually) periodic points}\label{S:periodic-density}

From now on we consider only compact metrizable spaces.
The reason is that we use results known only in such spaces
(results on cantoroids) and results which do not work
without compactness (Lemma~\ref{intr.pts} and the fact that
transitive maps in compact metrizable spaces have dense set of transitive points).

The following two lemmas are first steps towards the proof of
the dichotomy for transitive maps, see \theoremCref{}.

\begin{lemma}\label{perpt.freeint}
Let $X$ be a compact metrizable space with a free arc $A$ and let $f:X\to X$ be a transitive map.
Assume that $f$ has a periodic point $x_0$ in $A$. Then the set of
periodic points of $f$ is dense in $X$.
\end{lemma}

\begin{proof}
For a transitive map, the set of periodic points is either nowhere dense or dense.
So it is sufficient to show that the
periodic points of $f$ are dense in $A$. Assume, on the contrary,
that there is a free interval $J\subseteq A$ with $J\cap \Per(f)=\emptyset$.
The space $X$ has no isolated point (otherwise, due to transitivity of $f$, it would be finite).
By Lemma~\ref{intr.pts}, the set $\tr(f)$ of transitive points of $f$
is dense in $X$ and hence dense in $J$.
Every transitive point is
clearly recurrent, so $\Rec(f)$ is dense in $J$. By Theorem~\ref{per.rec.pts2}
we have
$$
J=\overline{\Rec(f)}\cap J=\left[\Rec(f)\cup \overline{\Per(f)}\right]\cap J
= \Rec(f)\cap J.
$$
Thus every point of $J$ is recurrent and hence no point of $J$ is eventually mapped
to $x_0$.

For $n\geq 0$ put $J_n=f^n(J)$ and consider the set $Y=\bigcup_{n=0}^{\infty}J_n$.
Then $x_0\not\in Y$. Further, $Y$ is $f$-invariant
and is dense in $X$ by transitivity of $f$. Therefore the restriction
of $f$ to $Y$ is also transitive. Since $Y$ contains a nonempty open connected set $J$,
$Y$ has only finitely many connected components, say $Y_0,\dots,Y_{p-1}$,
they are cyclically permuted by $f$ and the restriction of $f^p$ to each of them
is topologically transitive. We may assume that $Y_0$ is the component of
$Y$ containing $J$. So $Y_0$ is a connected space with a free interval $J$.

We claim that $J$ is in fact a disconnecting interval for $Y_0$.
According to Proposition~\ref{disc.intrs}(i) it is sufficient to find a point $x\in J$ such that
$Y_0\setminus\{x\}$ is not connected. We show that every $x\in J$ works. To this end
fix $x\in J$. Then $x\ne x_0$ and we may assume that $x\prec x_0$.
Notice that, since $X$ is Hausdorff, the compact set $[x,x_0]$ is closed in $X$, hence
$X\setminus [x,x_0]$ is open.
Then $(x,x_0)$ and $X\setminus [x,x_0]$ are disjoint open sets in $X$ whose union
contains $Y_0\setminus \{x\}$. It follows easily that $(x,x_0)\cap Y_0$ and $Y_0\setminus [x,x_0]$
form a separation of $Y_0\setminus\{x\}$, so $Y_0\setminus\{x\}$ is not connected.

By Lemma~\ref{denseperiod} the periodic points of $f^p|_{Y_0}$ are dense in
$Y_0$. Consequently, the periodic points of $f$ are dense in $J$ which is a contradiction.
\end{proof}

\begin{lemma}\label{S.impl.per}
Let $X$ be a compact metrizable space with a free interval $J$ and let
$f:X\to X$ be a transitive map. Assume that there is a nonempty closed nowhere
dense invariant set $S\subseteq X$ such that $S\cap J\neq\emptyset$. Then
the set of periodic points of $f$ is dense in $X$.
\end{lemma}
\begin{proof}
Let $\pi$
denote the quotient map which collapses the set $S$ into a point, call it $s$,
and denote the corresponding quotient space of $X$ by $Y$. Obviously,
the underlying decomposition of $X$ is upper semi-continuous which
implies that $Y$ is compact and metrizable. Since $S$ is invariant
for~$f$, we have an induced dynamics on $Y$. More precisely, there
is a continuous map $g:Y\to Y$ with $g\circ\pi=\pi\circ f$.
The restriction of $\pi$ to $X\setminus S$ is a homeomorphism onto
$Y\setminus\{s\}$ (it is a continuous bijection and it is open since
every open subset of $X\setminus S$ is saturated). So, since
$S$ is nowhere dense, to prove the density of $\Per(f)$ in $X$
it is sufficient to show that $\Per(g)$ is dense in $Y$.

Choose a free arc $A$ in $J\subseteq X$
whose intersection with $S$ is just one point, an end point of $A$.
Then $B=\pi(A)$ is a free arc in $Y$ whose one end point is $s$.
Further, $(Y,g)$ is a transitive system, being
a factor of $(X,f)$. Finally,
$s\in B$ is a fixed point of $g$.
By Lemma~\ref{perpt.freeint} applied to $g$
we obtain that $\Per(g)$ is dense in $Y$.
\end{proof}

If $f$ is a transitive map on an infinite compact metrizable space
with dense periodic points then it can happen that
there are no points in the system which are eventually periodic
but not periodic. Such an example can be found in \cite[Section~3]{DY}.
(It is a so-called ToP-system, see also the end of Section~\ref{S:dense-periodicity}).
The following lemma shows that under the additional assumption that
the space has a free interval the eventually periodic points do exist.

\begin{lemma}\label{dense.ev.per}
Let $X$ be a compact metrizable space
with a free interval $J$ and let $f:X\to X$ be a
transitive map with dense set of periodic points.
Then the set of all eventually periodic points of
$f$ which are not periodic is dense in $X$.
\end{lemma}
\begin{proof}
By Lemma~\ref{finite.DI}, $f$ has a terminal
regular periodic decomposition $\mathcal D=(D_0,\dots ,D_{m-1})$.
There is $i\in \{0,\dots ,m-1\}$ such that
$\Int (D_i)\cap J\neq \emptyset$. Then $D_i$
is a compact metrizable space with a free interval $J'$
(a subinterval of $J$). Moreover, by Lemma~\ref{termin.RPD},
the map $g=f^m|_{D_i}$ is totally transitive on $D_i$.
Since $D_i$ is regular closed and $\Per(f)$ is dense in $X$
we get that $\Per(g)$ is dense in $D_i$.
By Lemma~\ref{TT+DPimplWM}, $g$ is weakly
mixing. Now choose a periodic point $p$ of
$g$ satisfying $\card (\Orb_{g}
(p)\cap J')\geq 3$. Let $p_1\prec p_2\prec
p_3$ be members of $\Orb_{g}(p)$ lying in
$J'$. We show that the set of points non-periodic
for $g$ which are eventually mapped into
$\Orb_{g}(p)$ (and hence are eventually
periodic for $g$) is dense in $J'$. To this end
we prove that any open interval $V\subseteq J'\setminus
\Orb_{g}(p)$ contains such a point.
Since $g$ is
weakly mixing, there exists $k\in \NNN$ with
$g^k(V)\cap (p_j,p_{j+1})\neq \emptyset$
for $j=1,2$. Then the set $g^k(V)$, being connected,
contains $p_j$ for some $j\in \{ 1,2,3\}$.
Thus there exists a point $e\in V$ which is mapped
 by $g^k$ into $\Orb_{g}(p)$.
Since $e\notin
\Orb_{g}(p)$, $e$ is eventually periodic but not periodic for $g$.

We have shown that points in $J'$ which are
eventually periodic but not periodic for $g$
are dense in $J'$. Trivially, in this statement we may replace
$g$ by $f$.

To finish the proof fix a nonempty open set $U$ in $X$.
By transitivity of $f$ there is  an open
subinterval $K$ of $J'$ and integers $0\le j\le n$ such that
$f^j(K)\subseteq U$
and $f^n(K)\subseteq J'$.
The interval $f^n(K)$ is non-degenerate
and so it contains a point $y_0$ which is eventually periodic
but not periodic for $f$.
Since $f^{(n-j)}(U)\supseteq f^n(K)$, there is $x_0\in U$
with $f^{(n-j)}(x_0)=y_0$. The point $x_0$ is eventually periodic but not periodic
for $f$
which proves the density of such points in $X$.
\end{proof}

%%%%%%%%%%%%%%%%%%%%%%%%%%%%%%%%%%%%%%%%%%%%%%%%%%%%%%%%%%%%%%%%%%%%%%%%%%%%%%%
%%%%%%%%%%%%%%%%%%%%%%%%%%%%%%%%%%%%%%%%%%%%%%%%%%%%%%%%%%%%%%%%%%%%%%%%%%%%%%%
%%%%%%%%%%%%%%%%%%%%%%%%%%%%%%%%%%%%%%%%%%%%%%%%%%%%%%%%%%%%%%%%%%%%%%%%%%%%%%%
%%%%%%%%%%%%%%%%%%%%%%%%%%%%%%%%%%%%%%%%%%%%%%%%%%%%%%%%%%%%%%%%%%%%%%%%%%%%%%%
%%%%%%%%%%%%%%%%%%%%%%%%%%%%%%%%%%%%%%%%%%%%%%%%%%%%%%%%%%%%%%%%%%%%%%%%%%%%%%%
\section{Minimality and a trichotomy
for minimal sets (proofs of Theorems~A and B)}\label{S:minimality}

We embark on the proof of Theorems~A and B.

\begin{lemma}\label{free.intrs.1}
Let $X$ be a second countable Hausdorff space and $J$ be a nonempty subset of $X$. Then
$J$ is a free interval if and only if it is open, connected, locally homeomorphic to $\RRR$ and it is not a circle.
\end{lemma}
\begin{proof} One implication is trivial, for the converse one use the fact that
 if $J$ is connected and locally homeomorphic to $\RRR$ then, being a
 connected one-dimensional topological manifold, it is either a circle or an open interval.
\end{proof}

Notice that if $J$ is a free interval in a circle $X$ then always there is  maximal (with respect to the inclusion)
free interval $J^*$ containing $J$ (in fact, $J^*$ is the circle minus a point). However, if the complement of $J$ is not a singleton then
$J^*$ is not unique. If $X$ is not a circle, the following lemma shows that the things work better.

\begin{lemma}\label{free.intrs}
Let $X$ be a continuum which is not a circle.
\begin{enumerate}
  \item[(a)] If $J_1,J_2$ are free intervals in $X$ then either they are disjoint or their union is again a free interval.
  \item[(b)] Two maximal (with respect to the inclusion) free intervals of $X$ either are disjoint or coincide.
  \item[(c)] If $J$ is a free interval in $X$ then it is a subset of a (unique) maximal
    free interval.
\end{enumerate}
\end{lemma}
\begin{proof} (a)
Assume that $J_1\cap J_2\ne \emptyset$. Then $J=J_1\cup J_2$ is a connected open subset of $X$.
Suppose that $J$ is a circle. Then $J$ is closed in $X$ and so it is a clopen subset of $X$. Hence $X=J$
by connectedness of $X$, a contradiction. In view of Lemma~\ref{free.intrs.1} it remains to show that
$J$ is locally homeomorphic to $\RRR$.
Fix $x\in J$; we may assume that $x\in J_1$.
Since $J_1$ is a free interval there is a neighborhood $U$ of $x$ (in the topology of $X$) which is homeomorphic to $\RRR$
and is a subset of $J_1$, hence a subset of $J$. So, $U$ is a neighborhood of $x$ (in the topology of $J$) homeomorphic to $\RRR$.

(b)
Indeed, if $J_1$ and $J_2$
are maximal free intervals in $X$ with $J_1\cap J_2\neq\emptyset$ then, by (a),
$J_1\cup J_2$ is a free interval in $X$. By maximality of both $J_1$ and $J_2$
we have $J_1=J_1\cup J_2=J_2$.

(c)
Let $J^*$ denote the union of all free intervals containing $J$. Then, analogously as in (a), $J^*$ is a free interval. Obviously it
is a maximal free interval containing $J$. Uniqueness follows from~(b).
\end{proof}

\begin{lemma}\label{transit.homeo}
Let $X$ be a continuum with a free interval and let $f$ be a
transitive homeomorphism on $X$. Then $X$ is a circle and $f$ is conjugate to an irrational rotation.
\end{lemma}
\begin{proof}
If $X$ is a circle use Lemma~\ref{irr.rot}.
Supposing that $X$ is not a circle, we are going to find a space $Y$ homeomorphic to a circle
and a transitive homeomorphism $g$ on it with a fixed point, which will contradict Lemma~\ref{irr.rot}.
The system $(Y,g)$ will be obtained as a factor of $(X,f^n)$ for some $n>0$.

By Lemma~\ref{free.intrs}(c) there is a maximal free interval
and since $f$ is a homeomorphism, every maximal free interval is mapped
onto such an interval. Then transitivity of $f$, in view of Lemma~\ref{free.intrs}(b),
gives that there are only finitely many pairwise disjoint maximal free intervals, say $J_1,\dots,J_n$ ($n\ge 1$),
and $f$ permutes them in a periodic way.
Since $X$ is compact and the free intervals $J_1,\dots,J_n$ are pairwise disjoint,
the set $X\setminus\bigcup_{i=1}^n J_i$ is nonempty. It is closed, $f$-invariant and, by transitivity of $f$, nowhere dense.

Denote by $\mathcal D$ the decomposition of $X$ whose elements
are the singletons $\{x\}$ with $x\in J_1$ and the (closed) set $X\setminus J_1$.
Obviously $\mathcal D$ is upper semi-continuous and so the decomposition
space $Y=X/\mathcal D$ is a (metrizable) continuum. Denote by $\pi$ the quotient
map $X\to Y$. Clearly, $\pi(J_1)$ is a free interval in $Y$.
Since $Y\setminus\pi(J_1)$ is a singleton,
the space $Y$, being a one-point compactification of an open interval, is a circle.

The map $f^n$ is a homeomorphism of $X$ and both $J_1$ and $X\setminus J_1$
are $f^n-$ invariant. Consequently, there exists a homeomorphism $g$ of $Y$
with $g\circ\pi=\pi\circ f^n$.
Moreover, an elementary argument gives that $f^n|_{J_1}$ is transitive
(alternatively one can use (RPD6) for the regular periodic decomposition
$(J_1,\dots,J_n)$ for $f$ restricted to $\bigcup_i J_i$).
Therefore also $g|_{\pi(J_1)}$, being conjugate to $f^n|_{J_1}$, is transitive.
It follows that the homeomorphism $g$ is transitive on the circle $Y$.
However, the singleton $\pi(X\setminus J_1)$ is a fixed point for $g$
which contradicts Lemma~\ref{irr.rot}.
\end{proof}

\begin{corollary}\label{transit.homeo.disc}
A continuum with a disconnecting interval does not admit a transitive homeomorphism.
\end{corollary}
\begin{proof}
Consider an appropriate iterate of the homeomorphism and apply Lemma~\ref{transit.homeo}.
\end{proof}

Of course in this corollary it is substantial that we speak on homeomorphisms
(the tent map is a transitive map on a continuum with a disconnecting interval).

\begin{proposition}\label{minimal.continua}
Let $X$ be a continuum with a free interval $J$ and let $f$ be a minimal map
on $X$. Then $X$ is a circle and $f$ is conjugate to
an irrational rotation.
\end{proposition}
\begin{proof}
In view of Lemma~\ref{transit.homeo} it is sufficient to show that $f$ is one-to-one.
We proceed by contradiction. Suppose that
$f(x)=f(y)$ for some $x\ne y$. By Lemma~\ref{dense.backward} there are $n\in\NNN$ and
$a\in J$ with $f^n(a)=x$. Now use surjectivity of $f$ to find $b\in X$ with $f^n(b)=y$.
There is $k\ge n+1$ such that
$c=f^{k}(a)=f^{k}(b)\in J$. Notice that $a,b,c$ are pairwise distinct (because
$x\ne y$ and $f$ has no periodic point) and $a,c\in J$.
Moreover, by Lemma~\ref{tot.minimal}, $f^{k}$ is minimal.

Before proceeding further realize that the following claim holds.

\medskip

\noindent
\textit{Claim}. Let $Y$ be a topological space and $g:Y\to\RRR$ be a continuous map.
Let $U,V$ be two disjoint nonempty open sets in $Y$ such that $g(U)$ and $g(V)$ overlap
(i.e. there is an open interval $L$ with $L\subseteq g(U)\cap g(V)$).
Then there is a nonempty open set $G$ in $Y$ with $g(G)\subseteq g(Y\setminus G)$.

\medskip
The proof of the claim is obvious, just set $G=U\cap g^{-1}(L)$.

\medskip

To finish the proof of the proposition choose an open interval $a\in A\subseteq J$ and
an open neighborhood $B\ni b$ such that $A,B$ are disjoint and
both $f^k(A)$ and $f^k(B)$ are subsets of $J$.
The set $A\setminus \{a\}$ consists of two disjoint open intervals $A_1,A_2$
with $a$ being their common limit point.

Both $f^k(A_1)$ and $f^k(A_2)$ contain an open interval having the point $c$ as its left or right end point.
Moreover, by Lemma~\ref{feebly.open},
we have $\Int f^k(B')\neq\emptyset$ for every neighborhood $B'\subseteq B$ of $b$.
It follows that either the sets $f^k(A_1),f^k(A_2)$ overlap or one of them overlaps with $f^k(B)$.
In any case we may use the claim,
with $Y=A\cup B$ and $g=f^k|_Y$,
to find an open redundant set for $f^k$ which contradicts the minimality of $f^k$, see Lemma~\ref{red.open.set}.
\end{proof}

Now we are ready to prove \theoremABrefs{}. For reader's convenience we
repeat the statements here.

\theoremA
\begin{proof}
Since $f$ is minimal and $X$ has a component with nonempty interior
(namely the one containing $J$), $X$ has only finitely many components
$C_0,\dots, C_{n-1}$ and they are cyclically permuted by $f$. We may assume
that $C_0\supseteq J$ and $f(C_i)=C_{i+1(\text{mod}\,n)}$ for all $0\leq i\leq n-1$.
Since $f^n|_{C_0}$ is minimal and $C_0$ is a
continuum with a free interval $J$, we have, by Proposition~\ref{minimal.continua},
that $C_0$ is a circle and $f^n|_{C_0}$ is conjugate
to an irrational rotation. Now fix $i\in\{1,\dots,n-1\}$.
Since $f^n:C_0\to C_0$ is a homeomorphism, it follows that also
$f^i:C_0\to C_i$ is a homeomorphism and hence it is a conjugacy between
$f^n|_{C_0}$ and $f^n|_{C_i}$. So all the maps $f^n|_{C_i}$ ($i=0,\dots,n-1$)
are conjugate to the same irrational rotation which completes
the proof.
\end{proof}

\theoremB
\begin{proof}
Trivially the three cases are mutually exclusive.
Assume first that $M\cap J$ contains an arc. Then $M$ is a compact metrizable space
with a free interval admitting a minimal map $f|_M$. By \theoremAref{},
$M$ is a disjoint union of finitely many circles and so we have (3).

Now we assume that $M\cap J$ contains no arc and we show that in this case
either (1) or (2) holds. Clearly, $M\cap J$ is totally disconnected.
If $M\cap J$ has an isolated point then $M$ has an isolated point
and so $M$, being a minimal set, is finite. So assume that $M\cap J$ is dense
in itself. We prove (2).
The set $M$, being an infinite minimal set, has no isolated point.
Thus to show that $M$ is a cantoroid we only need to prove that the union of
all degenerate components of $M$ is dense in $M$.
By Lemma~\ref{unique.preimage}, the set $R$ of all points $x\in M$ which have a unique
preimage in $M$ is residual in $M$. By Lemma~\ref{image.residual} all the images $f^n(R)$ of $R$
are residual in $M$. Since $M\cap J$ is a nonempty open subset of $M$
it necessarily contains a point $z_0$ from the residual (in $M$) set $\bigcap_{n=0}^{\infty}f^n(R)$.
Since $z_0\in M\cap J$, it is a degenerate component of $M$ and since
$z_0\in \bigcap_{n=0}^{\infty}f^n(R)$, it has a unique backward orbit $\{z_n\}_{n=0}^\infty$ with respect to  $f|_M$.
By Lemma~\ref{dense.backward}, this orbit is dense in $M$.
Since $z_0$ is a component of $M$, the singleton
$z_n=(f|_M)^{-n}(z_0)$, being a union of components of $M$, is itself a component of $M$.
Hence degenerate components of $M$ are dense in $M$ and so $M$ is a cantoroid.

It remains to show that $M$ is nowhere dense in $X$. Suppose, on the contrary,
that the (closed) set $M$ contains a set $U$ which is nonempty and open in $X$.
Fix an arc $A\subseteq J$ containing a nonempty open subset
of $M$ (in the topology of $M$). Since $M$ is minimal for $f$, there is $n\in\NNN$ such that
$\bigcup_{i=0}^nf^i(A)\supseteq M\supseteq U$. Since the sets $f^i(A)$ are
closed there is $j$ such that $f^j(A)\supseteq V$ for some nonempty open set $V\subseteq U$.
The set $f^j(A)$, being a continuous image of an arc, is locally connected.
Therefore we may assume that $V$ is connected. Since $M$ is minimal and intersects
$J$, there is $k\in\NNN$ with $f^k(V)\cap J\neq\emptyset$. Since the
components of $M\cap J$ are singletons, it follows that
$f^k(V)$ is also a singleton. This implies that $M$ is finite, contradicting
our assumptions.
\end{proof}

%%%%%%%%%%%%%%%%%%%%%%%%%%%%%%%%%%%%%%%%%%%%%%%%%%%%%%%%%%%%%%%%%%%%%%%%%%%%%%%
%%%%%%%%%%%%%%%%%%%%%%%%%%%%%%%%%%%%%%%%%%%%%%%%%%%%%%%%%%%%%%%%%%%%%%%%%%%%%%%
%%%%%%%%%%%%%%%%%%%%%%%%%%%%%%%%%%%%%%%%%%%%%%%%%%%%%%%%%%%%%%%%%%%%%%%%%%%%%%%
%%%%%%%%%%%%%%%%%%%%%%%%%%%%%%%%%%%%%%%%%%%%%%%%%%%%%%%%%%%%%%%%%%%%%%%%%%%%%%%
%%%%%%%%%%%%%%%%%%%%%%%%%%%%%%%%%%%%%%%%%%%%%%%%%%%%%%%%%%%%%%%%%%%%%%%%%%%%%%%
\section{Transitivity and dense periodicity}\label{S:dense-periodicity}

\theoremCref{} will follow from Lemma~\ref{transit.homeo} and from three other lemmas,
each of which will deal with some `partial' dichotomy
for transitive maps (we call them partial because Theorem~C called
a ``Dichotomy for transitive maps'' combine all of them). In this section we prove the first of these three lemmas.

\begin{lemma}\label{dichotomy1}
Let $X$ be a compact metrizable space with a free interval $J$ and let
$f:X\to X$ be a transitive map. Then exactly one of the following
two statements holds.
\begin{enumerate}
\item[(1)] The periodic points of $f$ form a dense subset of $X$.
\item[(2)] The space \rotationX{}.
\end{enumerate}
\end{lemma}
\begin{proof}
By Lemma~\ref{intr.pts} the set $\intr(f)$ of intransitive points of $f$ is
either empty or dense. If it is empty, then $(X,f)$ is minimal and \theoremAref{}
gives (2).

So assume that the set $\intr(f)$ is dense in $X$ and fix $x_0\in
J\cap\intr(f)$. Denote by $S$ the closure of the orbit of $x_0$
under $f$. Clearly, $S$ is a nonempty closed invariant set. Furthermore,
we have $S\subseteq\intr(f)$. Indeed, if $S$ contained a transitive
point $y_0$ then, being invariant, it would contain the orbit of $y_0$
and hence it would be dense in $X$, contradicting the fact that $x_0$
is not transitive. So $S\subseteq\intr(f)$ and, since $\tr(f)$ is
dense in $X$, the (closed) set $S$ is nowhere dense in $X$.
Lemma~\ref{S.impl.per} then gives (1).
\end{proof}

By \cite{DY}, a dynamical system $(X,f)$ is called a \emph{ToM-system} if $X$ is a compact metrizable space, $f$ is transitive,
not minimal, and every point of $X$ is either transitive or minimal
(a point is called \emph{minimal} if it belongs to a minimal set).
If in such a system every point is either transitive or periodic,
the system is called a \emph{ToP-system}.
Notice that if $(X,f)$ is a ToM-system then $X$ has no isolated point.
(Otherwise, by transitivity of $f$, the system would be just one periodic orbit,
so it would be minimal.)

\begin{corollary}\label{no.ToP.sys}
Let $X$ be a compact metrizable space with a free interval $J$. Then
there are no ToM-systems on $X$.
\end{corollary}
\begin{proof}
Suppose that $(X,f)$ is a ToM-system.
Since $f$ is transitive and not minimal, by Lemma~\ref{dichotomy1} it has a dense set of periodic
points. By Lemma~\ref{dense.ev.per}, $f$ has a dense set of points
which are eventually periodic but not periodic, a contradiction.
\end{proof}

%%%%%%%%%%%%%%%%%%%%%%%%%%%%%%%%%%%%%%%%%%%%%%%%%%%%%%%%%%%%%%%%%%%%%%%%%%%%%%%
%%%%%%%%%%%%%%%%%%%%%%%%%%%%%%%%%%%%%%%%%%%%%%%%%%%%%%%%%%%%%%%%%%%%%%%%%%%%%%%
%%%%%%%%%%%%%%%%%%%%%%%%%%%%%%%%%%%%%%%%%%%%%%%%%%%%%%%%%%%%%%%%%%%%%%%%%%%%%%%
%%%%%%%%%%%%%%%%%%%%%%%%%%%%%%%%%%%%%%%%%%%%%%%%%%%%%%%%%%%%%%%%%%%%%%%%%%%%%%%
%%%%%%%%%%%%%%%%%%%%%%%%%%%%%%%%%%%%%%%%%%%%%%%%%%%%%%%%%%%%%%%%%%%%%%%%%%%%%%%
\section{Transitivity and topological entropy}\label{S:entropy}

The following result should be well known but we have not found it explicitly
and so we include a proof.

\begin{lemma}\label{horseshoe}
Let $X$ be a compact metrizable space and $f:X\to X$ a continuous
map. Assume that there exist nonempty closed pairwise disjoint
subsets $A_1,\dots ,A_m$ of $X$ and a positive integer $k$
such that for every $i\in \{ 1,\dots ,m\}$ there is $p_i\in
\NNN$ with
$$
\card \{ j\in \{ 1,\dots ,m\}:\, f^{p_i}(A_i)
\supseteq A_j\} \geq k.
$$
Then for $p=\max \{ p_1,\dots ,p_m\}$ we have
$h(f)\geq (1/p)\log k$.
\end{lemma}
\begin{proof}
Call a finite sequence $\overline{s}=(s_0,\dots ,s_n)$ ($n\geq 0$) of elements
of the set $\{ 1,\dots ,m\}$ realizable
if $f^{p_{s_i}}(A_{s_i})\supseteq A_{s_{i+1}}$ for $0\leq i\leq
n-1$. For such a sequence $\overline{s}$  the set
%$A_{s_0,\dots ,s_n}=\bigcap _{l=0}^{n}f^{-\sum _{i=0}^{l-1}p_{s_i}}(A_{s_l})$
$$
 A_{\overline{s}}
 =\{
 x\in A_{s_0}:\ f^{p_{s_0}}(x)\in A_{s_1}, f^{p_{s_0}+p_{s_1}}(x)\in A_{s_2}, \dots, f^{p_{s_0}+\dots + p_{s_{n-1}}}(x)\in A_{s_n}
 \}
$$
is obviously nonempty and so one can fix $x_{\overline{s}}
\in A_{\overline{s}}$. For $n\in \NNN$ put
$$
 E_n = \{x_{\overline{s}}:\ \overline{s}=(s_0,\dots,s_n) \text{ is realizable} \}.
$$
By our assumptions $\card E_n\geq m k ^n\geq k ^{n+1}$. Choose $0<\varepsilon<
\min_{i\ne j} \dist (A_i,A_j)$. It follows that the set $E_n$ is
$(np+1,\varepsilon)$-separated for $f$. Consequently, by Bowen's definition of topological entropy,
$$
h(f)\geq \limsup_{n\to \infty}\frac{1}{np+1}\log (\card E_n)\geq
\limsup_{n\to \infty}\frac{n+1}{np+1}\log k =\frac{1}{p}\log k,
$$
as desired.
\end{proof}

\begin{lemma}\label{weak-general}
Let $X$ be a compact metrizable space with a free interval $J$. Then every weakly mixing map
$f$ on $X$ has positive topological entropy.
\end{lemma}
\begin{proof}
We prove the following
assertion:
\begin{enumerate}
\item[(A)] If a connected set $C\subseteq X$ contains
 three distinct points $x_1,x_2,x_3\in J$, then it contains at least one of the (non-degenerate) subarcs of $J$
 whose end points are in $\{x_1,x_2,x_3\}$.
\end{enumerate}
We may assume that $x_1\prec x_2\prec x_3$. Suppose that
$C$ contains neither $\langle x_1,x_2\rangle$ nor $\langle x_2,x_3\rangle$.
Fix $y_1\in \langle x_1,x_2\rangle \setminus C$ and $y_2\in \langle
x_2,x_3\rangle \setminus C$. Then the sets $(y_1,y_2)$ and $X\setminus \langle y_1,y_2\rangle$
constitute a separation of $C$ in $X$ which is a contradiction.

Now we prove the lemma. Fix three disjoint arcs $A_1,A_2,A_3$ in $J$. Then the
set $J\setminus \bigcup _{i=1}^3A_i$ has four components $U_j$ ($1\leq j\leq 4$)
which are open in $X$. Since $f$ is weakly mixing, there is
$p\in \NNN$ satisfying $f^{p}(A_i)\cap U_j\neq \emptyset$ for every $i,j$.
The Assertion (A) now secures that for every $i$ the connected set $f^{p}(A_i)$ contains at
least two of the sets $A_1,A_2,A_3$. Thus, by Lemma~\ref{horseshoe},
$h(f)\geq (1/p) \log 2>0$.
\end{proof}

We get the second partial dichotomy for transitive maps.

\begin{lemma}\label{dichotomy2}
Let $X$ be a compact metrizable space with a free interval $J$ and let
$f:X\to X$ be a transitive map. Then exactly one of the following
two statements holds.
\begin{enumerate}
\item[(1)] The topological entropy of $f$ is positive.
\item[(2)] The space \rotationX{}.
\end{enumerate}
\end{lemma}
\begin{proof}
Assume that $f$ does not satisfy condition (2). Then, by Lemma~\ref{dichotomy1},
the set of periodic points of $f$ is dense in $X$.
By Lemma~\ref{finite.DI} there is a terminal regular periodic decomposition
$\mathcal D=(D_0,\dots,D_{m-1})$ for $f$. We may assume that $D_0$ intersects $J$ and so
$D_0$ is compact metrizable with a free interval. By Lemma~\ref{termin.RPD}
the map $f^m|_{D_0}$ is totally transitive.
Since the periodic points of $f$ are dense in $X$ and
$D_0$ is regular closed in $X$, it follows that
the map $f^m|_{D_0}$ has dense periodic
points and, by Lemma~\ref{TT+DPimplWM}, is weakly  mixing.
It follows from Lemma~\ref{weak-general} that $h(f^m|_{D_0})>0$ and so $h(f)>0$.
\end{proof}

%%%%%%%%%%%%%%%%%%%%%%%%%%%%%%%%%%%%%%%%%%%%%%%%%%%%%%%%%%%%%%%%%%%%%%%%%%%%%%%
%%%%%%%%%%%%%%%%%%%%%%%%%%%%%%%%%%%%%%%%%%%%%%%%%%%%%%%%%%%%%%%%%%%%%%%%%%%%%%%
%%%%%%%%%%%%%%%%%%%%%%%%%%%%%%%%%%%%%%%%%%%%%%%%%%%%%%%%%%%%%%%%%%%%%%%%%%%%%%%
%%%%%%%%%%%%%%%%%%%%%%%%%%%%%%%%%%%%%%%%%%%%%%%%%%%%%%%%%%%%%%%%%%%%%%%%%%%%%%%
%%%%%%%%%%%%%%%%%%%%%%%%%%%%%%%%%%%%%%%%%%%%%%%%%%%%%%%%%%%%%%%%%%%%%%%%%%%%%%%
\section{Transitivity and strong mixing. Proof of Theorem~C}\label{S:mixing}

To prove \theoremCref{} we will need the following, already the third, partial
dichotomy for transitive maps.
(Recall that $f$ is relatively strongly mixing if $f$ has an RPD $(D_0,\dots, D_{m-1})$
such that $f^m|_{D_i}$ is strongly mixing for every $i$.)

\begin{lemma}\label{dichotomy3}
Let $X$ be a compact metrizable space with a free interval $J$ and let
$f:X\to X$ be a transitive map. Then exactly one of the following
two statements holds.
\begin{enumerate}
\item[(1)] The map $f$ is relatively strongly mixing.
\item[(2)] The space \rotationX{}.
\end{enumerate}
\end{lemma}
\begin{proof}
The proof is word by word the same as that of Lemma~\ref{dichotomy2}
only at the very end of it we use Lemma~\ref{TT+DPimplSM}
instead of Lemma~\ref{weak-general}, to get that $f^m|_{D_0}$ is strongly mixing.
Then every $f^m|_{D_i}$, being a factor of $f^m|_{D_0}$ by (RPD3), is also strongly mixing.
\end{proof}

\theoremC{}

\begin{proof}
Assume that (2) does not hold. Then $f$ is relatively strongly mixing by
Lemma~\ref{dichotomy3}, has positive entropy by Lemma~\ref{dichotomy2}
and has dense periodic points by Lemma~\ref{dichotomy1}. Finally, $f$ is non-invertible by
Lemma~\ref{transit.homeo}.
\end{proof}

%%%%%%%%%%%%%%%%%%%%%%%%%%%%%%%%%%%%%%%%%%%%%%%%%%%%%%%%%%%%%%%%%%%%%%%%%%%%%%%
%%%%%%%%%%%%%%%%%%%%%%%%%%%%%%%%%%%%%%%%%%%%%%%%%%%%%%%%%%%%%%%%%%%%%%%%%%%%%%%
%%%%%%%%%%%%%%%%%%%%%%%%%%%%%%%%%%%%%%%%%%%%%%%%%%%%%%%%%%%%%%%%%%%%%%%%%%%%%%%
%%%%%%%%%%%%%%%%%%%%%%%%%%%%%%%%%%%%%%%%%%%%%%%%%%%%%%%%%%%%%%%%%%%%%%%%%%%%%%%
%%%%%%%%%%%%%%%%%%%%%%%%%%%%%%%%%%%%%%%%%%%%%%%%%%%%%%%%%%%%%%%%%%%%%%%%%%%%%%%
\newcommand{\downf}[1]{\stackrel{{#1}}{\searrow}} %{\rightharpoondown}}
\newcommand{\Pgf}[2]{\operatorname{P}_{#1,#2}}

\section*{Appendix~1}\label{S:appendix}

As we promised in Section~\ref{S:periodic-recurrent},
we prove here Theorem~\ref{per.rec.pts2}.
For maps $f$ and $g$, each of them being defined on a (possibly different) subset of $X$
and with values in $X$, put
$$
 \Pgf{g}{f}=\Fix(f)\cup \Fix(g) \cup
 \bigcup_{k\ge 1} \Fix(g^k\circ f).
$$
For an interval map $\varphi$ we will write $[\alpha,\beta]\downf{\varphi} [\gamma,\delta]$ if
$\varphi([\alpha,\beta])=[\gamma,\delta]$, $\varphi(\alpha)=\delta$ and $\varphi(\beta)=\gamma$.
In the proof of the next lemma we will use the following simple observation:
If $\varphi([\alpha',\beta'])\supseteq[\gamma,\delta]$, $\varphi(\alpha')=\delta$ and $\varphi(\beta')=\gamma$
then there are $\alpha'\le \alpha < \beta\le \beta'$ with
$[\alpha,\beta]\downf{\varphi} [\gamma,\delta]$.

\begin{lemma}\label{L:recper1}
Let $a<d$ and $b,c\in(a,d)$ be reals and let
$f:[a,b]\to[a,d]$ and $g:[c,d]\to[a,d]$ be interval maps such that
$f(a)=d$, $g(c)=a$ and $g(d)\in f([a,b])$.
Then $\Pgf{g}{f}\ne\emptyset$.
\end{lemma}
\begin{proof}
Assume that neither $f$ nor $g$ has a fixed point. Then $g(d)<d$ and,
by the assumption, $\min_{x\in[a,b]} f(x)\le g(d)$.
By replacing $b$ by appropriate $a<b'<b$
(and replacing $f$ by $f'=f|_{[a,b']}$), if necessary,
we may assume that $f(x)>f(b)=g(d)$ for every $x\in[a,b)$.
Analogously by replacing $c$ by a larger number, if necessary, we may assume that
$a = g(c) < g(y)$ for every $y\in (c,d]$. See Figure~1  %\ref{Fig:recper}.
(the choice $b<c$ in the figure does not play any role in the proof).

Since $f$ and $g$ have no fixed points, it holds
$$
 f(x)>x \quad\text{and}\quad g(y)<y
 \qquad
 \text{for every } x\in[a,b] \text{ and } y\in[c,d].
$$
Thus $\min_{y\in[c,d]} (y-g(y))>0$ and hence there is $m\ge 1$ such that
$$
 d_0:=d
 \ > \
 d_1:=g(d_0)
 \ > \
 \dots
 \ > \
 d_m:=g(d_{m-1})
 \ \ge \  c \
 \ > \
  d_{m+1}:=g(d_m).
$$
If $d_{m+1}=a$ then $a\in \Fix(g^{m+1}\circ f)$ and the proof is finished. So assume that
$a<d_{m+1}$.

 \begin{figure}[ht]
  \label{Fig:recper}
  \begin{center}
    \epsffile{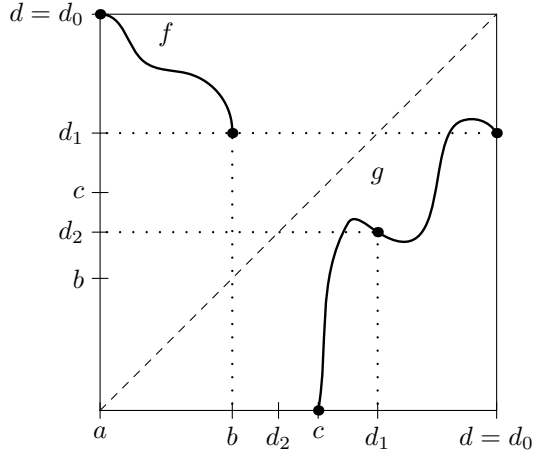}
   \caption{Illustration  for the case when $m=1$.}
  \end{center}
 \end{figure}

For $i\ge 0$ put $g_i = g^i\circ f$.
Since $f([a,b])\supseteq [d_1,d_0]$ and
$g([d_{i},d_{i-1}])\supseteq [d_{i+1},d_{i}]$ for $i=1,\dots,m$, we get
$$
 g_i([a,b])\supseteq [d_{i+1},d_i],
 \qquad\text{for }i=0,\dots,m
$$
(of course, $g_i([a,b])$ means the $g_i$-image of the intersection of $[a,b]$ with
the domain of $g_i$).
So, by the observation above the lemma, one can find $a_i, b_i$ such that
$$
 a=a_0\le a_1\le \dots\le a_m \ < \ b_m \le \dots \le b_1\le b_0= b
 \qquad
 \text{and}
 \qquad
 [a_i,b_i] \downf{g_i} [d_{i+1},d_i].
$$
Moreover, $g([d_{m+1},d_m])=g([c,d_m]) \supseteq [a,d_{m+1}]$
and so there are $a_{m+1}<b_{m+1}$ in $[a_m,b_m]$ such that
$$
 [a_{m+1},b_{m+1}] \downf{g_{m+1}} [a,d_{m+1}].
$$
Then $g_m(b_{m+1}) = c$ since $g^{-1}(a)=\{c\}$.

Put
$$
 p = \max\{0\le i \le m+1:\ b_i<g_i(b_i)\}.
$$
Since $b<f(b)$ can be written as $b_0 < g_0(b_0)$, $p\ge 0$ exists. Moreover,
$p\le m$ since $b_{m+1}>a =g_{m+1}(b_{m+1})$.
Now $g_{p+1}(b_{p+1}) \le b_{p+1}$ (by the choice of $p$)
and, regardless of whether $p\le m-1$ or $p=m$, we get
$g_{p+1}(a_{p+1}) = d_{p+1}=g_p(b_p)>b_p>a_{p+1}$.
So
$g_{p+1}=g^{p+1}\circ f$ has a fixed point in $[a_{p+1},b_{p+1}]$.
Hence $\Pgf{g}{f}\ne\emptyset$.
\end{proof}

To simplify the things,
in the following two lemmas we will assume that a given free arc $A$
is (not only homeomorphic to $[0,1]$ but) exactly $[0,1]$.
This of course does not restrict generality and enables us to imagine
a part of the considered dynamics as an interval one.

\begin{lemma}\label{L:recper2}
Let $X$ be a compact metrizable space with a free arc $A=[0,1]$.
Let $f:A\to X$ and $g:X\to X$ be continuous maps such that $A\cap \Pgf{g}{f}=\emptyset$
and there are $0<x<y<1$ with $f(x),g(y)\in [x,y]$. Then the following are true:
\begin{enumerate}
    \item[(1)] If $f(x)\le g(y)$ then $f(0)\in (0,1)$.
    \item[(2)] If $f(x)>g(y)$ then $f(0)\in (0,1)$ or $g(f(0))\in (0,1)$.
\end{enumerate}
\end{lemma}
\begin{proof} (1)
Assume, on the contrary, that $f(0)\not\in (0,1)$. Let $v\in [0,x)$ be maximal such that
$f(v)\in\{0,1\}$; such $v$ exists since $f([0,x])$ is a connected set intersecting $(0,1)$
and $X\setminus (0,1)$.
By maximality of $v$ we get
$$
 f([v,x]) \subseteq [0,1].
$$
Then $f(v)=1$ since otherwise $[v,x]\subseteq f([v,x]) \subseteq [0,1]$ and
$f$ would have a fixed point in $[v,x]$.
Let $a\in[v,x]$ be maximal such that $f(a)=y$; then $f([a,x])\subseteq [a,y]$.
Finally let $c\in[a,y]$ be maximal such that $g(c)=a$ (it must exist since $g$ has no fixed point
in $[a,y]$ and $g(y)<y$);
then $g([c,y])\subseteq [a,y]$.
Denote $f'=f|_{[a,x]}$ and $g'=g|_{[c,y]}$. Since $f(x)\le g(y)<y=f(a)$ we get
$g'(y)\in f'([a,x])$.
Now Lemma~\ref{L:recper1}
applied to  $f',g'$
gives that $\Pgf{g'}{f'}\ne \emptyset$,
a contradiction with $\Pgf{g'}{f'}\subseteq \Pgf{g}{f}\cap A=\emptyset$.

(2) Now assume that $x\le g(y)<f(x)\le y$ and $f(0)\not\in (0,1)$. As in the proof of (1)
let $v\in[0,x)$ be maximal such that $f(v)\in\{0,1\}$. Then again $f(v)=1$ and
$[f(x),1]\subseteq f([v,x])\subseteq [0,1]$. So there is $x'\in (v,x)$ with $f(x')=y$.
Put $f'=g\circ f$. Then $f':A\to X$, $g:X\to X$, $x'<x\le g(y)\le y$ and $f'(x')=g(y)$. Moreover, $\Pgf{g}{f'}\cap A\subseteq \Pgf{g}{f}\cap A=\emptyset$.
The already proved case (1) applied to $f',g$ and $x',y$ gives that $g(f(0))=f'(0)\in (0,1)$.
\end{proof}

\begin{lemma}\label{L:recper3} Let $X$ be a compact metrizable space with a free
arc $A=[0,1]$. Let $f:X\to X$ be a continuous map such that
$\Per(f)\cap [0,1]=\emptyset$ and $\Rec(f)\cap(0,1)\ne\emptyset$. Then
$\Orb_f(0)\cap (0,1)\ne\emptyset$.
\end{lemma}
\begin{proof}
Let $x\in\Rec(f)\cap (0,1)$; then there are positive integers $m,n$
such that
$f^n(x)\in (0,1)$ and $f^{n+m}(x)$ is between $x$ and $f^n(x)$.

Assume first that $x<f^{n+m}(x)<f^n(x)$.
We are going to apply Lemma~\ref{L:recper2}(2) to the points $x,y=f^n(x)$ and the maps $F=f^n|_A, G=f^m$.
Since trivially $\Pgf{G}{F}\subseteq \Per(f)$, we have $\Pgf{G}{F}\cap [0,1] =\emptyset$.
Further, $0<x<G(y)<F(x)=y<1$. So Lemma~\ref{L:recper2}(2) gives
$F(0)\in(0,1)$ or $G(F(0))\in (0,1)$. So
$\Orb_f(0)\cap (0,1)\ne\emptyset$.

Now assume that $f^n(x)<f^{n+m}(x)<x$. Put $x'=f^n(x)$, $y'=x$ and $F=f^{m}|_A$, $G=f^n$.
Then $0<x'=G(y') < F(x') < y'<1$. Again $\Pgf{G}{F}\cap [0,1]=\emptyset$
and we can use Lemma~\ref{L:recper2}(2) as in the previous case.
\end{proof}

\begin{proof}[Proof of Theorem~\ref{per.rec.pts2}]
We only need to prove the inclusion
$$
\overline{\Rec(f)}\cap J \subseteq \left[\Rec(f) ~ \cup ~ \overline{\Per(f)}\right]\cap J.
$$
Assume that $x\in J$ is such that $x\in \overline{\Rec}(f)\setminus \overline{\Per}(f)$.
If $x$ is recurrent we are done. So assume that $x\not\in\Rec(f)$.
 Take arbitrarily small free interval $L\subseteq J$ containing $x$ which is disjoint with ${\Per}(f)$.
 Choose an orientation on $L$ which ensures that for some
 free arc $A=[x,b]\subseteq L$ there is a recurrent point of $f$ in $(x,b)$.
 By Lemma~\ref{L:recper3}, the orbit of $x$ intersects $(x,b)\subseteq L$.
 Since $L$ was arbitrarily small, $x$ is recurrent. This contradicts our assumption.
\end{proof}

%%%%%%%%%%%%%%%%%%%%%%%%%%%%%%%%%%%%%%%%%%%%%%%%%%%%%%%%%%%%%%%%%%%%%%%%%%%%%%%
%%%%%%%%%%%%%%%%%%%%%%%%%%%%%%%%%%%%%%%%%%%%%%%%%%%%%%%%%%%%%%%%%%%%%%%%%%%%%%%
%%%%%%%%%%%%%%%%%%%%%%%%%%%%%%%%%%%%%%%%%%%%%%%%%%%%%%%%%%%%%%%%%%%%%%%%%%%%%%%
%%%%%%%%%%%%%%%%%%%%%%%%%%%%%%%%%%%%%%%%%%%%%%%%%%%%%%%%%%%%%%%%%%%%%%%%%%%%%%%
%%%%%%%%%%%%%%%%%%%%%%%%%%%%%%%%%%%%%%%%%%%%%%%%%%%%%%%%%%%%%%%%%%%%%%%%%%%%%%%
\section*{Appendix~2}\label{S:appendix2}

We are going to show that in compact connected Hausdorff spaces, many natural definitions of a disconnecting
interval are equivalent. However, in more general spaces this is not the case, see Proposition~\ref{disc.intrs}.

Let $X$ be a topological space and $J$ be a free interval. Recall that
we assume that one of the two natural orderings on $J$
(induced by the usual orderings on a real interval) is chosen and denoted by $\prec$.
If $K$ is a subinterval of $J$ we define the (possibly empty) sets
$$
 J_K^-:=\{z\in J:\ z\prec y \text{ for all } y\in K\}
 \qquad
 \text{and}
 \qquad
 J_K^+:=\{z\in J:\ y\prec z \text{ for all } y\in K\}.
$$

If $X$ is a topological space we write $X=A|B$ to mean $X=A\cup B$, $A\ne\emptyset$, $B\ne\emptyset$,
$A\cap B=\emptyset$ and $A$ and $B$ are both open in $X$. In other words $X=A|B$ means $X=A\cup B$
where $A$ and $B$ are nonempty sets which are separated in $X$ (recall that $A$ and $B$
are \emph{separated} provided $\overline{A}\cap B=A \cap \overline{B}=\emptyset$).
Sometimes the union of two disjoint sets $A,B$ will be denoted by $A\sqcup B$.
We will need the following two lemmas whose proofs can be found for instance in
\cite[Proposition~6.3]{Nad} (cf. \cite[\S 46, II, Theorem~4]{Kur2}) and \cite[Lemma~6.4]{Nad}, respectively.

\begin{lemma}\label{separation}
Let $X$ be a connected topological space, $C$ be a connected subset of $X$ and $X\setminus C=A|B$. Then
$A\cup C$, $B\cup C$ are connected and, if $C$ is closed, then they are also closed.
\end{lemma}

\begin{lemma}\label{separation.cutpts}
Let $X$ be a connected topological space and $a,b\in X$ be such that
$$
 X\setminus \{a\} = A_1|A_2
 \qquad\text{and}\qquad
 X\setminus \{b\} = B_1|B_2 \,.
$$
If $a\in B_1$ and $b\in A_2$ then $B_2\cup\{b\} \subseteq A_2$.
\end{lemma}

From Lemma~\ref{separation} we immediately obtain the following observation which will be used repeatedly.
\begin{lemma}\label{separation.trivial}
Let $X$ be a connected topological space, $C$ be a connected subset of $X$ and $X\setminus C=A|B$. Then
\begin{enumerate}
  \item[(a)] $\partial{C}\cap\partial{A}\ne\emptyset$ and $\partial{C}\cap\partial{B}\ne\emptyset$;
  \item[(b)] if $C$ is open then $A,B$ are closed and $\partial{C}\cap{A}\ne\emptyset$, $\partial{C}\cap{B}\ne\emptyset$;
  \item[(c)] if $C$ is closed then $A,B$ are open and ${C}\cap\partial{A}\ne\emptyset$, ${C}\cap\partial{B}\ne\emptyset$.
\end{enumerate}
\end{lemma}

To shorten some formulations in the next lemma and in Proposition~\ref{disc.intrs}, we introduce the
following terminology. If $J$ is a free interval of a space $X$ and $K\subseteq J$ is an interval
such that $J\setminus K$ has two components, we say that $K$ is a \emph{bi-proper subinterval of $J$}.
Note that every point $c\in J$ is a (degenerate) bi-proper subinterval of $J$.

\begin{lemma}\label{separation-po-kuse}
Let $X$ be a connected Hausdorff topological space, $J$ be a free interval of $X$
and $K$ be a subinterval of $J$ such that $X\setminus K$ is not connected.
Then for every separation $X\setminus K=L_K|R_K$
the notation can be (and in what follows always will be) chosen in such a way that
\begin{equation}\label{EQ:separationLR}
 L_K\cap J=J_K^-
 \qquad\text{and}\qquad
 R_K\cap J=J_K^+ \,.
\end{equation}
Moreover, if $K$ is a bi-proper subinterval of $J$, then $J_K^-,J_K^+$ are nonempty and $L_K,R_K$ are
the (two) components of $X\setminus K$.
Regardless of whether $K$ is bi-proper or not, it holds
\begin{equation}\label{EQ:separationLR2}
 \overline{L_K}=\begin{cases}
   L_K\cup\{\inf K\} &\text{if } J_K^-\ne\emptyset
   \\
   L_K  &\text{otherwise}
 \end{cases}
 \qquad\text{and}\qquad
 \overline{R_K}=\begin{cases}
   R_K\cup\{\sup K\} &\text{if } J_K^+\ne\emptyset
   \\
   R_K  &\text{otherwise}.
 \end{cases}
\end{equation}
\end{lemma}
Before reading the proof we recommend to see Example~\ref{EX:separ-three-comps}.
It illustrates that if the subinterval $K$ is not bi-proper then $L_K,R_K$ are not uniquely determined by $K$
and they need not be connected.

\begin{proof}
First we prove (\ref{EQ:separationLR2}) assuming that the rest of the lemma is true.
Since $L_K$ and $R_K$ are separated we have $\overline{L_K}\setminus L_K\subseteq K \subseteq J$.
If $J_K^-$ is nonempty then, by (\ref{EQ:separationLR}) and taking into account that $J$ is open in $X$, we get $\overline{L_K}\cap J = J_K^- \cup\{\inf K\}$ and so
$\overline{L_K}=L_K\cup \{\inf K\}$. If $J_K^-=\emptyset$,  $\overline{L_K}\cap J = \emptyset$ and $\overline{L_K}=L_K$.
Similarly for $\overline{R_K}$.

It remains to prove the lemma without (\ref{EQ:separationLR2}).

Since $K$ disconnects $X$, there are $L_K$ and $R_K$ with $X\setminus K=L_K|R_K$.
The sets $J_K^-,J_K^+$ are connected (possibly empty); if some of these two sets is nonempty, then it is either
a subset of $L_K$ or a subset of $R_K$.
Distinguish four cases.

If $K=J$ then there is nothing to prove since \emph{every} separation $X\setminus K=L_K|R_K$ satisfies (\ref{EQ:separationLR}).
If $J_K^-=\emptyset\ne J_K^+$ then, by changing the notation of the sets $L_K$ and $R_K$ if necessary, we obtain (\ref{EQ:separationLR})
and the same argument works if $J_K^+=\emptyset\ne J_K^-$; in these two cases the first part of
the lemma just fixes the notation so that equation (\ref{EQ:separationLR}) hold. It remains to prove the lemma if $K$ is bi-proper.

So assume that $K$ is bi-proper (i.e. $J_K^-, J_K^+$ are nonempty) and $X\setminus K=L_K|R_K$. By changing the notation, if necessary, we may assume that $L_K\supseteq J_K^-$;
we are going to prove that $R_K\supseteq J_K^+$. Suppose, on the contrary, that $L_K\supseteq J_K^+$.
Then $\partial K\cap \partial R_K\subseteq J\cap \overline{R_K} \subseteq  \overline{J\cap R_K}=\emptyset$ (since $J$ is open and
$J\subseteq K\cup L_K$) which contradicts Lemma~\ref{separation.trivial}(a).

Once we know that $L_K\supseteq J_K^-$, $L_K\cap K=\emptyset$ and $L_K\cap J_K^+=\emptyset$ (since $R_K\supseteq J_K^+$), then taking into account that $J=J_K^-\sqcup K\sqcup J_K^+$
we get that $L_K\cap J=J_K^-$. Analogously $R_K\cap J=J_K^+$.

To show that $L_K$ is connected, suppose on the contrary that $L_K=A|B$. Since $J_K^-$ is connected, it is a subset of $A$ or $B$, say $J_K^-\subseteq A$.
We claim that $X=B|(A\cup K\cup R_K)$.

Now $B\cap J \subseteq B\cap (A\cup K\cup R_K)=\emptyset$. Since $J$ is open, $\overline{B}\cap J\subseteq \overline{B\cap J}$. Hence, $\overline{B}\cap J = \emptyset$.
Since $J$ is homeomorphic to an open interval and $K\subseteq J$ is bi-proper,
for $a=\inf K$ and $b=\sup K$ we have $K\subseteq [a,b]\subseteq J$.
Obviously, $[a,b]$ is compact and since $X$ is assumed to be Hausdorff, $[a,b]$ is a closed set in $X$.
Thus $\overline{K}\subseteq [a,b] \subseteq J$ and so $\overline{B}\cap \overline{K}\subseteq \overline{B}\cap J=\emptyset$.
Using this fact we get
$\overline{B}\cap (A\cup K\cup R_K) = \emptyset$ and $B\cap \overline{A\cup K\cup R_K} = \emptyset$.
This contradicts the connectedness of $X$ and so we have proved that $L_K$ is connected.
Similarly, $R_K$ is connected.
\end{proof}

\begin{lemma}\label{separation2}
Let $X$ be a connected Hasudorff topological space with a free interval $J$ and let $c\in J$ be a cut point of $X$.
Then
\begin{enumerate}
  \item[(a)] $X\setminus \{c\} = L_c|R_c$ where $L_c,R_c$ are the connected sets from Lemma~\ref{separation-po-kuse};
  \item[(b)] $\overline{L_c}=L_c\cup\{c\}$ and $\overline{R_c}=R_c\cup\{c\}$;
  \item[(c)] $\{c\}$ is closed in $X$ and $L_c$ and $R_c$ are open in $X$.
%\end{enumerate}
%Moreover, if $X$ is additionally assumed to be Hausdorff then
%\begin{enumerate}
  \item[(d)] each $z\in J$ is a cut point of $X$;
  \item[(e)] every bi-proper subinterval $K$ of $J$ cuts $X$ into exactly two components.
\end{enumerate}
\end{lemma}
\begin{proof}  (a) follows from Lemma~\ref{separation-po-kuse}.

(b) Using (\ref{EQ:separationLR2}) from Lemma~\ref{separation-po-kuse} we get  $\overline{L_c}=L_c\cup \{\inf\{c\}\} = L_c\cup\{c\}$ and, similarly,
$\overline{R_c}=R_c\cup\{c\}$.

(c) It follows from (b) that $R_c=X\setminus(L_c\cup\{c\})=X\setminus\overline{L_c}$ is open. Similarly $L_c$ is open and so $\{c\}=X\setminus(L_c\cup R_c)$
is closed.

(d)
Fix a point $z\in J\setminus\{c\}$, say $z\prec c$.
Then $[z,c)\subseteq J_c^- \subseteq  L_c$.
Observe also that, since $X$ is Hausdorff, the compact set $[z,c]$ is closed and so we have $\overline{(z,c]} = \overline{[z,c)}=[z,c]$.

Put
$L=L_c \setminus [z,c)$ and $R = (z,c] \sqcup R_c$.
Then $X=L\sqcup \{z\} \sqcup R$ and we only need to show that $L$ and $R$ are separated.
First, $\overline{R}=[z,c]\cup \overline{R_c}=[z,c)\sqcup \overline{R_c}$.
Further, we show that $\overline{L}=L\cup\{z\}$. This follows from three facts. First, by (c),
$\overline{L}\subseteq \overline{L_c}=L_c\cup\{c\}$. Second, no point $x\in (z,c]$ belongs to $\overline{L}$
since $J_z^+$ is a neighborhood of $x$ disjoint with $L$. Third, $z\in\overline{L}$ since every neighborhood of $z$
intersects $J_z^-\subseteq L$. Then for the set $L_c=L\sqcup [z,c)$ we get
$\overline{L_c}=\overline{L} \cup [z,c] = \overline{L}\sqcup (z,c]$. So
$\overline{L}=\overline{L_c}\setminus (z,c]$. To summarize,
$\overline{L}=\overline{L_c}\setminus [z,c)$ and $\overline{R}=[z,c)\sqcup \overline{R_c}$.
Hence $\overline{L}\cap R
= \left(\overline{L_c}\setminus (z,c] \right)
  \cap
  \left((z,c]\sqcup R_c  \right)
\subseteq  \overline{L_c}\cap R_c=\emptyset$ and similarly ${L}\cap \overline{R} \subseteq {L_c}\cap \overline{R_c}=\emptyset$,
i.e. $L,R$ are separated.

(e) Let $K$ be a bi-proper subinterval of $J$. To show that $X\setminus K$ has exactly two components it is sufficient,
by Lemma~\ref{separation-po-kuse}, to show that $X\setminus K$ is not connected. If $K$ is degenerate this follows from (d).
So assume that $K$ is non-degenerate and choose a point $a$ from the interior of $K$. By (d), $a$ is a cut point of $X$ and,
by Lemma~\ref{separation-po-kuse}, the sets $L_a,R_a$ form a separation of $X\setminus\{a\}$. It follows that $(X\setminus K)\cap L_a$
and $(X\setminus K)\cap R_a$ form a separation of $X\setminus K$ and so $X\setminus K$ is not connected.
\end{proof}

\begin{lemma}\label{free.intrs.2}
Let $X$ be a connected compact Hausdorff space and $J$ be a free interval of $X$.
Assume that $X\setminus J$ is disconnected.
Then $X\setminus J$ has exactly two components $L_J,R_J$,
the set $\partial J$ is nowhere dense, has exactly two components  $\mathcal{L}_0,\mathcal{R}_0$
and the notation can be chosen in such a way that
$\mathcal{L}_0\subseteq L_J$, $\mathcal{R}_0\subseteq R_J$ and for each $c\in J$ it holds
\begin{equation}\label{EQ:L0R0}
 \overline{J_c^-} = \mathcal{L}_0\sqcup J_c^- \sqcup\{c\}, \quad
 \overline{J_c^+} = \{c\} \sqcup J_c^+ \sqcup \mathcal{R}_0
 \quad\text{and}\quad
 X\setminus\{c\} = (L_J\sqcup J_c^-) ~|~ (J_c^+\sqcup R_J)
  \,.
\end{equation}
\end{lemma}

\begin{proof}
Fix a separation $X\setminus J = L_J|R_J$ whose existence is guaranteed by Lemma~\ref{separation-po-kuse}.

Now we show the existence of $\mathcal{L}_0,\mathcal{R}_0$ with the required properties.
Since $J$ is a free interval in $X$ we may identify it with the \emph{real interval} $(0,1)$.
 For $0<\delta<1$ put $\mathcal{L}_\delta = (0,\delta)$ and $\mathcal{R}_\delta = (1-\delta,1)$.
 Then $\mathcal{L}_0=\bigcap_{\delta>0} {\overline{\mathcal{L}_\delta}}$, being the intersection of a nested family of nonempty compact connected sets
 in a compact Hausdorff space, is a nonempty compact connected set.
 By changing the notation of $\mathcal{L}_\delta,\mathcal{R}_\delta$, if necessary, we may assume that $\mathcal{L}_0\subseteq L_J$.
 Further, $\mathcal{L}_0\subseteq \overline{J}\setminus J$; indeed, the inclusion $\mathcal{L}_0\subseteq \overline{J}$ is trivial and a point from $J$ cannot belong to $\mathcal{L}_0$
 since $J$ is open.
 Similarly for $\mathcal{R}_0=\bigcap_{\delta>0} {\overline{\mathcal{R}_\delta}}$. Hence $\mathcal{L}_0\cup \mathcal{R}_0\subseteq \overline{J}\setminus J$.
 Also the converse inclusion holds since for every $\delta>0$ we have
 $\overline{J}= \overline{\mathcal{L}_\delta} \cup [\delta,1-\delta] \cup \overline{\mathcal{R}_\delta}$ and so $\overline{J}\setminus J\subseteq \bigcap_\delta (\overline{\mathcal{L}_\delta} \cup \overline{\mathcal{R}_\delta})=\mathcal{L}_0\cup \mathcal{R}_0$.
 We have thus proved that
 $$
  \partial{J} = \overline{J}\setminus J = \mathcal{L}_0\cup \mathcal{R}_0.
 $$
 By Lemma~\ref{separation.trivial}(b), each of the sets $L_J,R_J$ intersects $\partial{J}=\mathcal{L}_0\cup \mathcal{R}_0$. Since $\mathcal{L}_0\subseteq L_J$
 we must have $\mathcal{R}_0\subseteq R_J$. It follows that $\mathcal{L}_0,\mathcal{R}_0$ are disjoint which implies that they are the two components of $\partial J$.
 The boundary $\partial{J}$ is nowhere dense since $J$ is open.

 Now we prove that $L_J,R_J$ are connected.
 Suppose that $L_J$ is not connected and fix a separation $A|B$ of $L_J$. Then the connected set $\mathcal{L}_0$ lies in one of the sets $A,B$, say $\mathcal{L}_0\subseteq B$.
 Then $A,J$ are separated and so  $X=A|(B\cup J\cup R_J)$ is not connected, which is a contradiction.
 Analogously we can show that $R_J$ is connected.

Fix $c\in J$. It remains to show (\ref{EQ:L0R0}).
For every sufficiently small $\delta$ we have $J_c^-\supseteq \mathcal{L}_\delta$, so $\overline{J_c^-}\supseteq \bigcap_{\delta>0} \overline{\mathcal{L}_\delta} = \mathcal{L}_0$.
 Hence $\overline{J_c^-} \supseteq \mathcal{L}_0\sqcup J_c^- \sqcup\{c\}$.
 Further, for every sufficiently small $\delta>0$ the set
 $\overline{\mathcal{L}_\delta}\cup J_c^- \cup\{c\}=\overline{\mathcal{L}_\delta}\cup [\delta,c]$ is closed and contains $J_c^-$, hence it contains $\overline{J_c^-}$.
 Thus $\overline{J_c^-} \subseteq \bigcap_{\delta>0} \left( \overline{\mathcal{L}_\delta}\cup J_c^- \cup\{c\} \right) = \mathcal{L}_0\sqcup J_c^- \sqcup\{c\}$.
 We have proved the first formula in (\ref{EQ:L0R0}); analogously for the second one.
 The third formula follows from the first two ones and the fact that $L_J,R_J$ are separated and closed.
\end{proof}

\begin{proposition}\label{disc.intrs}
Let $X$ be a connected Hausdorff topological space and $J$ be a free interval in $X$. Consider the following conditions:
\begin{enumerate}
  \item[(1a)] $X\setminus J$ is not connected;
  \item[(1b)] $X\setminus J$ has exactly two components;

  \smallskip

  \item[(2a)] there is a point $x\in J$ which cuts $X$;
  \item[(2b)] every point $x\in J$ cuts $X$ into exactly two components (i.e. $J$ is a disconnecting interval, see Definition~\ref{D:disc});

  \smallskip

  \item[(3a)] there is a bi-proper subinterval $K$ of $J$ such that the set $X\setminus K$ is not connected;
  \item[(3b)] for every bi-proper subinterval $K$ of $J$ the set $X\setminus K$ has exactly two components;

  \smallskip

  \item[(4a)] there is a subinterval $K$ of $J$ such that the set $X\setminus K$ is not connected;
  \item[(4b)] for every subinterval $K$ of $J$ the set $X\setminus K$ has exactly two components.
\end{enumerate}
Then the following hold:
\begin{enumerate}
  \item[(i)]  $(4b)\Rightarrow(1b)\Rightarrow(1a)\Rightarrow(4a)$ and $(4b)\Rightarrow(2a)\Leftrightarrow (2b)\Leftrightarrow(3a)\Leftrightarrow (3b) \Rightarrow(4a)$.
%  \item[(ii)] If $X$ is Hausdorff then additionally the four conditions $(2a)-(3b)$ are equivalent.
  \item[(ii)] If $X$ is a \emph{compact} connected Hausdorff space then all the eight conditions $(1a)-(4b)$ are equivalent.
\end{enumerate}
In the case $(i)$ no other implication, except of those which follow by transitivity, is true.
\end{proposition}

\begin{proof}
(i)
It is sufficient to show (2a)$\Leftrightarrow$(2b)$\Leftrightarrow$(3a)$\Leftrightarrow$(3b) since all other needed implications are trivial.
Since the implications (3b)$\Rightarrow$(2b)$\Rightarrow$(2a)$\Rightarrow$(3a) are also trivial, it remains to prove that (3a)$\Rightarrow$(3b).
However, (2a)$\Rightarrow$(3b) holds by Lemma~\ref{separation2} and so it is sufficient to prove
(3a)$\Rightarrow$(2a).
To this end, let $K\subseteq J$ be a bi-proper subinterval of $J$ such that $X\setminus K$ is not connected.
If $K$ is degenerate there is nothing to prove. Otherwise put $a=\inf K$ and $b=\sup K$. Then, since $X$ is Hausdorff,
the same argument as in the last paragraph of the proof of
Lemma~\ref{separation-po-kuse} gives that $[a,b]$ is a closed set in $X$ and so $\overline{K}=[a,b]$.
By Lemma~\ref{separation-po-kuse}, $X\setminus K=L_K|R_K$ where $L_K,R_K$ are connected and contain $J_K^-,J_K^+$, respectively.
We are going to prove that $a$ is a cut point of $X$. We do not know whether $a$ belongs to $L_K$ or to $K$, therefore
denote $L_K'=L_K\setminus\{a\}$ and $K'=K\setminus\{a\}$.
Obviously, $\overline{K'}=\overline{K}=[a,b]$ and $\overline{L_K'}=\overline{L_K}=L_K\cup\{a\}$ by Lemma~\ref{separation-po-kuse}.
It follows that $\overline{K'}\cap \overline{L_K'}=\overline{K}\cap \overline{L_K}=\{a\}$ and hence
$\overline{K'}\cap {L_K'}={K'}\cap \overline{L_K'}=\emptyset$.
Using these facts and the fact that $L_K$ and $R_K$ are separated we immediately get that
 $X\setminus\{a\} = L_K' | (K'\sqcup R_K)$.

(ii)
Assume that $X$ is also compact. We need to prove that (4a)$\Rightarrow$(4b).
The implications (1a)$\Rightarrow$(1b) and (1a)$\Rightarrow$(2a) follow from Lemma~\ref{free.intrs.2}. Moreover, we know
that (2a) is equivalent with (3b). So to finish the proof it is sufficient to show that (4a)$\Rightarrow$(1a) and that
(1b) and (3b) together imply (4b).

First we show that (4a)$\Rightarrow$(1a).
So assume (4a). We may identify $J$ with the real interval $(0,1)$.
Write $X\setminus K = L_K|R_K$ as in Lemma~\ref{separation-po-kuse}.
We are going to show that $K'=K\cup J_K^-$ also cuts $X$.
This is trivial if $J_K^-=\emptyset$, so assume that
$J_K^-\ne\emptyset$. Put $a=\inf K>0$. Take a decreasing sequence $a>a_1>a_2>\dots$ of real numbers converging to $0$.
Put $K_n=(a_n,a] \cup K$ and $\mathcal{L}_{n}=L_K\setminus (a_n,a]$. Then $\mathcal{L}_{n}$ is nonempty and trivially for the interval $K_n\subseteq J$
we have $X\setminus K_n = \mathcal{L}_{n}|R_K$. Since $\mathcal{L}_n \cap J = J_{K_n}^-$, the sets $\mathcal{L}_n$ and $R_K$ in this separation correspond to the sets $L_{K_n}$ and $R_{K_n}$
from Lemma~\ref{separation-po-kuse}. Hence $\overline{\mathcal{L}_{n}}=\mathcal{L}_{n}\cup\{\inf K_n\}=\mathcal{L}_{n}$.
Since $(\mathcal{L}_{n})_{n=1}^\infty$ is a nested sequence of nonempty compact sets, the set $\mathcal{L}=\bigcap_n \mathcal{L}_{n}$ is nonempty.
Now, using the fact that $K'=\bigcup_n K_n$, we have that $X\setminus K' = \mathcal{L}|R_K$, i.e. $K'$ cuts $X$.
Once we know that $K'=K\cup J_K^-$ cuts $X$, by an analogous argument we get that also $J=K'\cup J_{K'}^+$ cuts $X$, i.e. we get (1a).

To show that (1b) and (3b) together imply (4b), fix a subinterval $K$ of $J$. We need to show that $X\setminus K$ consists of two components.
This is trivial if $K=J$ (use (1b)) or if $K$ is bi-proper (use (3b)). It remains to consider the case when $J_K^-=\emptyset$ and $J_K^+\ne\emptyset$
or conversely. Suppose we are in the former case. By (1b), $X\setminus J = L_J|R_J$, where $L_J$ and $R_J$ are connected. We claim that $L_J$ and $J_K^+ \sqcup R_J$ are the two components
of $X\setminus K$. Indeed, the sets $L_J$ and $J_K^+ \sqcup R_J$ are separated (use that for $c\in K$, $\overline{J_K^+}\subseteq \overline{J_c^+}$ and Lemma~\ref{free.intrs.2}
gives $\overline{J_c^+}\cap L_J=\emptyset$). Further, we know that $L_J$ is connected.
Finally, the sets $J_K^+$ and $R_J$ are connected and not separated
(if $d\in J_K^+$ then $\overline{J_K^+}\supseteq\overline{J_d^+}$ and, by Lemma~\ref{free.intrs.2}, $\overline{J_d^+}\cap R_J \supseteq \mathcal{R}_0\ne\emptyset$).
Hence also the set  $J_K^+ \sqcup R_J$ is connected.

To finish the proof of the proposition we need to show that in the case (i) no other implication, except of those which follow by transitivity, is true.
It is easy to see that this requires three counterexamples. We collect them below (see Examples~\ref{EX:separ-three-comps}--\ref{EX:3b-1a}).
\end{proof}

\begin{example}[\emph{In Proposition~\ref{disc.intrs}(i), (1a)$\not\Rightarrow$(1b)}]\label{EX:separ-three-comps}
Consider the topologist's sine curve with the interval of convergence replaced by just two points of it.
If $J$ is the maximal free interval of this space then $X\setminus J$ consists of three degenerate components.
\end{example}

\begin{example}[\emph{In Proposition~\ref{disc.intrs}(i), (1b)$\not\Rightarrow$(2a)}]\label{EX:1b-2a}
Consider the (non-compact) subspace $X$ of the Euclidean plane defined by
$$
 X = J\sqcup C_{-1}\sqcup C_1
$$
where $C_i=[-1,1]\times \{i\}$ for $i\in\{-1, 1\}$ and $J$ is the graph
of the function $x\mapsto x\cdot \sin(1/1-\abs{x})$, $x\in (-1,1)$.
Then $X$ is connected, $J$ is a free interval and $X\setminus J=C_{-1}\cup C_1$ has two components.
However, no point $x\in X$ is a cut point of $X$.
\end{example}

\begin{example}[\emph{In Proposition~\ref{disc.intrs}(i), (3b)$\not\Rightarrow$(1a)}]  \label{EX:3b-1a}
Let $X=(0,1)$ or $X=[0,1)$ be a subset of the real line. Then $J=(0,1)$ is a disconnecting interval of $X$ and $X\setminus J$ is connected.
\end{example}

In Proposition~\ref{disc.intrs} we have assumed that $X$ is Hausdorff.
Without this assumption we would not have the equivalence of the four conditions (2a), (2b), (3a), (3b).
For instance the space in the following example satisfies (2a), (3a) but does not satisfy (2b), (3b).

\begin{example}\label{EX2:separ-three-comps}
 Let $X$ be the line with two origins, i.e. the factor space of $\RRR\times \{-1,1\}$ obtained by identifying $x \times (-1)$ with $x \times 1$ for every $x\ne 0$.
 Let $p:\RRR\times \{-1,1\} \to X$ be the corresponding factor map.
 Then $J=p((-1,1)\times 1)$ is a free interval containing a cut point of $X$, e.g. $c=p(1/2 \times 1)$. But $z=p(0\times 1)\in J$ is not a cut point of $X$.
 So we have (2a) but not (2b). Also, $K=p((-1/2,1/2)\times 1)$ is  a bi-proper subinterval of $J$ such that $X\setminus K$ is not connected but has precisely
 three components. So we have (3a) but not (3b).
\end{example}

%\end{large}


\begin{thebibliography}{999999999}


\bibitem[ALM00] {ALM} LL. Alsed\`{a}, J. Llibre, M. Misiurewicz,
       \textit{Combinatorial dynamics and entropy in dimension one}, Second edition,
       Advanced Series in Nonlinear Dynamics, 5, World Scientific Publishing Co., Inc., River Edge, NJ, 2000.

\bibitem[AKLS99] {AKLS} Ll. Alsed\`{a}, S. Kolyada, J. Llibre, \mL . Snoha,
                \textit{Entropy and periodic points for transitive maps},
                Trans. Amer. Math. Soc. \textbf{351}~(1999), no. 4, 1551--1573.

\bibitem[Ba01]{Baldwin01} S. Baldwin,
  \textit{Entropy estimates for transitive maps on trees},
  Topology \textbf{40}~(2001), no.~3, 551--569.

\bibitem[Ba97] {B}  J. Banks,
              \textit{Regular periodic decompositions for topologically transitive maps},
              Ergodic Theory Dynam. Systems \textbf{17}~(1997), no.~3, 505--529.

\bibitem[BC92] {BC}  L. S. Block, W. A. Coppel,
    \textit{Dynamics in one dimension}, Lecture Notes in Mathematics, 1513. Springer-Verlag, Berlin, 1992.

\bibitem[BDHSS09] {BDHSS}  F. Balibrea, T. Downarowicz, R. Hric, \mL{}. Snoha, V. \v Spitalsk\'y,
    \textit{Almost totally disconnected minimal systems},
    Ergodic Theory Dynam. Systems \textbf{29}~(2009), no.~3, 737--766.


\bibitem[BHS03] {BHS}  F. Balibrea, R. Hric, \mL{}. Snoha,
    \textit{Minimal sets on graphs and dendrites},
    Dynamical systems and functional equations (Murcia, 2000),
    Internat. J. Bifur. Chaos Appl. Sci. Engrg. \textbf{13}~(2003), no.~7, 1721--1725.



\bibitem[Bl84] {Blo84}  A. M. Blokh,
 \textit{On transitive mappings of one-dimensional branched manifolds. (Russian)},
  Differential-difference equations and problems of mathematical physics (Russian),
  3--9, 131, Akad. Nauk Ukrain. SSR, Inst. Mat., Kiev, 1984.

\bibitem[Bl86] {Blo}  A. M. Blokh,
   \textit{Dynamical systems on one-dimensional branched manifolds. I (Russian)},
   Teor. Funktsii Funktsional. Anal. i Prilozhen. \textbf{46}~(1986), 8--18;
   translation in J. Soviet Math. \textbf{48}~(1990), no.~5, 500--508.

\bibitem[CE80] {CE80}  P. Collet, J. P. Eckmann,
   \textit{Iterated maps on the interval as dynamical systems},
   Progress in Physics, 1. Birkh\"auser, Boston, Mass., 1980.

\bibitem[dMvS93] {dMvS} W. de Melo, S. van Strien,
   \textit{One-dimensional dynamics}, Ergebnisse der Mathematik und ihrer Grenzgebiete (3)
   [Results in Mathematics and Related Areas (3)], 25, Springer-Verlag, Berlin, 1993.


\bibitem[D32] {D32} A. Denjoy,
              \textit{Sur les courbes definies par les \'equations diff\'erentielles \`a la surface du tore} (French),
              J. Math. Pures Appl. \textbf{11}~(1932), 333--375.


\bibitem[DY02]{DY} T. Downarowicz, X. Ye,
  \textit{When every point is either transitive or periodic},
  Colloq. Math. \textbf{93}~(2002), no.~1, 137--150.

\bibitem[HKO11] {HKO}
    G. Hara\'nczyk, D. Kwietniak and P. Oprocha,
    \textit{A note on transitivity, sensitivity and chaos for graph maps},
    J. Difference Eq. Appl. \textbf{17}~(2011), no.~10, 1549--1553.

\bibitem[Il00]{I} A. Illanes,
  \textit{A characterization of dendrites with the periodic-recurrent property},
  Topology Proc. \textbf{23}~(1998), Summer, 221--235.

\bibitem[Ka88] {Kaw} K. Kawamura,
    \textit{A direct proof that each Peano continuum with a free arc admits no expansive homeomorphisms},
    Tsukuba J. Math. \textbf{12}~(1988), no. 2, 521--524.

\bibitem[Ki58] {Kin} S. Kinoshita,
                 \textit{On orbits of homeomorphisms}, Colloq. Math. \textbf{6}~(1958), 49–-53.

\bibitem[KS97] {KS} S. Kolyada, \mL. Snoha,
              \textit{Some aspects of topological transitivity -- a survey},
              Iteration theory (ECIT 94) (Opava), 3--35, Grazer Math. Ber., \textbf{334},
              Karl Franzens-Univ. Graz, Graz, 1997.

\bibitem[KST01] {KST} S. Kolyada, \mL. Snoha, S. Trofimchuk,
               \textit{Noninvertible minimal maps},
               Fundamenta Mathematicae \textbf{168}~(2001), 141--163.

\bibitem[Ku68] {Kur2} K. Kuratowski,
               \textit{Topology. Vol. II},
               Academic Press, New York-London,
               Pa\'nstwowe Wydawnictwo Naukowe Polish Scientific Publishers, Warsaw, 1968.

\bibitem[Kw11] {Kwi11} D. Kwietniak,
               \textit{Weak mixing implies mixing},
               preprint, 2011.

\bibitem[Ma11] {Mal11} P. Mali\v ck\'y,
  \textit{Backward orbits of transitive maps}
  J. Difference Equ. Appl., to appear.

\bibitem[MS07] {MShi} J. Mai, E. Shi,
     \textit{The nonexistence of expansive commutative group actions on Peano continua having free dendrites},
     Topology Appl. \textbf{155}~(2007), no. 1, 33--38.

\bibitem[MS09] {MS} J.-H. Mai, S. Shao,
  \textit{$\overline R=R\cup\overline P$ for graph maps}
  J. Math. Anal. Appl. \textbf{350}~(2009), no.~1, 9--11.


\bibitem[Na92]{Nad}
  S. B. Nadler,
  \textit{Continuum theory. An introduction},
  Monographs and Textbooks in Pure and Applied Mathematics, 158,
  Marcel Dekker, Inc., New York, 1992.


\bibitem[Si92]{Silv}
 S. Silverman,
 \textit{On maps with dense orbits and the definition of chaos},
 Rocky Mountain J. Math. \textbf{22}~(1992), no.~1, 353--375.

\bibitem[SW09] {SW} E. Shi, S. Wang,
  \textit{The ping-pong game, geometric entropy and expansiveness for group actions on Peano continua having free dendrites},
  Fund. Math. \textbf{203}~(2009), no. 1, 21--37.

\bibitem[XYZH96]{Warsaw96} J. C. Xiong, X. D. Ye, Z. Q. Zhang, J. Huang,
  \textit{Some dynamical properties of continuous maps on the Warsaw circle} (Chinese),
  Acta Math. Sinica (Chin. Ser.) \textbf{39}~(1996), no.~3, 294--299.

\bibitem[Ye92]{Ye}
  X. D. Ye,
  \textit{$D$-function of a minimal set and an extension of Sharkovskii's theorem to minimal sets},
  Ergodic Theory Dynam. Systems \textbf{12}~(1992), no.~2, 365--376.

\bibitem[ZPQY08]{Warsaw08} G. Zhang, L. Pang, B. Qin, K. Yan,
  \textit{The mixing properties on maps of the Warsaw circle},
  J. Concr. Appl. Math. \textbf{6}~(2008), no.~2, 139--144.


\end{thebibliography}
\end{document}